\documentclass{amsart}

\usepackage{url,graphicx,amssymb,enumerate,stmaryrd, adjustbox, enumitem}
\usepackage[pagebackref,colorlinks,citecolor=blue,linkcolor=blue,urlcolor=blue,filecolor=blue]{hyperref}
\usepackage{microtype}
\usepackage[all]{xy}
\usepackage{dsfont}
\usepackage{aliascnt}
\usepackage[vcentermath]{youngtab}
\usepackage[dvipsnames, hyperref]{xcolor}
\usepackage{pinlabel}
\usepackage{etoolbox}
\usepackage{tikz} \usetikzlibrary{matrix,arrows}

\usepackage[hmargin=3cm,vmargin=3cm]{geometry}

\newcommand{\m}{\to}
\newcommand{\bk}{\mathbf k}

\providecommand{\HH}{\mathbf H}
\providecommand{\CC}{\mathbf C}

\providecommand{\PBC}{\ensuremath\mathrm{PBC}}

\providecommand{\SPB}{\ensuremath\mathrm{SPB}}
\providecommand{\MPB}{\ensuremath\mathrm{MPB}}
\providecommand{\FI}{\ensuremath\mathsf{FI}}
\providecommand{\VIC}{\ensuremath\mathsf{VIC}}

\providecommand{\CB}{\ensuremath\mathsf{CB}}
\providecommand{\SI}{\ensuremath\mathsf{SI}}

\providecommand{\Cat}{\ensuremath\mathsf{C}}
\providecommand{\A}{\mathcal{A}}

\providecommand{\C}{\ensuremath\mathsf{C}}

\providecommand{\hooklongrightarrow}{\lhook\joinrel\longrightarrow}

\providecommand{\Z}{\ensuremath\mathbb Z}
\providecommand{\Q}{\ensuremath\mathbb Q}

\DeclareMathOperator{\im}{im}
\DeclareMathOperator{\coker}{coker}

\DeclareMathOperator{\Hom}{Hom}
\DeclareMathOperator{\End}{End}

\DeclareMathOperator{\Aut}{Aut}

\DeclareMathOperator{\GL}{GL}

\DeclareMathOperator{\Sp}{Sp}

\DeclareMathOperator{\Mod}{Mod}

\DeclareMathOperator{\Lk}{Lk}
\DeclareMathOperator{\Ind}{Ind}

\newcommand{\U}{\mathfrak U}

\newcommand{\G}{\mathrm{G}}

\newcommand{\cA}{\mathcal{A}}\newcommand{\cB}{\mathcal{B}}
\newcommand{\cC}{\mathcal{C}}
\newcommand{\cF}{\mathcal{F}}

\newcommand{\cK}{\mathcal{K}}
\newcommand{\cM}{\mathcal{M}}
\newcommand{\cP}{\mathcal{P}}

\newcommand{\cW}{\mathcal{W}}

\newcommand{\I}{{\bf I}}

\definecolor{grey}{gray}{.5}

\numberwithin{thmcounter}{section}
\newaliascnt{thmauto}{thmcounter}

\newaliascnt{Defauto}{thmcounter}

\newaliascnt{exauto}{thmcounter}

\newaliascnt{lemauto}{thmcounter}

\newaliascnt{propauto}{thmcounter}

\newaliascnt{corauto}{thmcounter}

\newaliascnt{remauto}{thmcounter}

\newaliascnt{convauto}{thmcounter}

\newtheorem{atheorem}{Theorem}

\newtheorem{theorem}[thmauto]{Theorem}
\newtheorem{lemma}[lemauto]{Lemma}
\newtheorem{proposition}[propauto]{Proposition}
\newtheorem{corollary}[corauto]{Corollary}
\theoremstyle{definition}
\newtheorem{definition}[Defauto]{Definition}

\newtheorem{remark}[remauto]{Remark}

\title{Quantitative representation stability over linear groups}
\author{Jeremy Miller}
\thanks{Jeremy Miller was supported in part by National Science Foundation grant DMS-1709726.}
\address{Department of Mathematics, Purdue University, USA}
\email{jeremykmiller@purdue.edu}
\author{Jennifer C. H. Wilson}
\address{Department of Mathematics, Stanford University, USA}
\email{jchw@stanford.edu}
\date{\today}

\begin{document}

\begin{abstract}

We introduce a technique for proving quantitative representation stability theorems for sequences of representations of certain finite linear groups over a field of characteristic zero. In particular, we prove a vanishing result for higher syzygies  of $\VIC$ and $\SI$-modules, which can be thought of as a weaker version of a regularity theorem of Church-Ellenberg \cite[Theorem A]{CE} in the context of $\FI$-modules.   We apply these techniques to the rational homology of congruence subgroups of mapping class groups and congruence subgroups of automorphism groups of free groups. This partially resolves a question raised by Church and Putman--Sam \cite[Remark 1.8]{PS}. We also prove new homological stability results for mapping class groups and automorphism groups of free groups with twisted coefficients.

\end{abstract}
\maketitle

\tableofcontents


\section{Introduction}

Putman--Sam \cite{PS} introduced techniques for proving \emph{representation stability} results in the sense of Church--Ellenberg--Farb \cite{CEF} for sequences of representations of several families of finite linear groups. They applied their tools to prove stability results for the homology groups of congruence subgroups of mapping class groups  and automorphism groups of free groups. 
In this paper, we introduce new techniques that allow us to establish explicit stable ranges. Moreover, our methods do not require that we work with homology groups which are finitely generated. These stronger results come at the cost of working with field coefficients of characteristic zero.

\subsection{Stability for congruence subgroups} The study of representation stability concerns the following framework:  fix a sequence of groups with inclusions \[\G_0 \hooklongrightarrow \G_1 \hooklongrightarrow  \G_2 \hooklongrightarrow  \G_3 \hooklongrightarrow  \cdots \] such as symmetric groups $S_n$, general linear groups $\GL_n(\bk)$, or symplectic groups $\Sp_{2n}(\bk)$. Fix a commutative ring $R$. Let $\{\cA_n\}$ be a sequence of $R[\G_n]$--modules with the data of $\G_n$--equivariant maps $\cA_n \m \cA_{n+1}$. The sequence $\{\cA_n\}$ is said to have \emph{generation degree} $\leq d$ if, for all $n \geq d$, the $R[\G_{n+1}]$--module generated by the image of $\cA_n$ is all of $\cA_{n+1}$. Informally, we say that the sequence $\{\cA_n\}$ stabilizes if its generation degree is finite. In this paper we also discuss a related notion called \emph{presentation degree}.

The main examples of spaces that we consider are classifying spaces of congruence subgroups of mapping class groups and congruence subgroups of automorphism groups of free groups. Let $\Mod(\Sigma_{g,r})$ denote the mapping class group of $\Sigma_{g,r}$, the compact orientable surface of genus $g$ with $r$ boundary components. The mapping class group acts on $H_1(\Sigma_{g,r})$. For $r \leq 1$, this action preserves the symplectic intersection form and so we get a map $$\Mod(\Sigma_{g,r}) \m \Sp_{2g}(\Z)$$ to the group $\Sp_{2g}(\Z)$ of symplectomorphisms of $\Z^{2g}$. 
Reducing modulo $p$ gives a map $$\Mod(\Sigma_{g,r}) \m \Sp_{2g}(\Z/p\Z)$$ and we denote the kernel by $\Mod(\Sigma_{g,r},p)$. This group is often called the \emph{level $p$ congruence subgroup of $\Mod(\Sigma_{g,r})$.} For $r=0$, the classifying space of this group has the homotopy type of the moduli stack of smooth genus $g$ complex curves with full level $p$ structure. For $r \leq 1$, the homology groups $H_i(\Mod(\Sigma_{g,r},p);R)$ have the structure of a $R[\Sp_{2g}(\Z/p\Z)]$--module. For $r=1$, the inclusions of surfaces $\Sigma_{g,1} \hookrightarrow \Sigma_{g+1, 1}$ induce $\Sp_{2g}(\Z/p\Z)$--equivariant maps $$H_i(\Mod(\Sigma_{g,1},p);R) \m H_i(\Mod(\Sigma_{g+1,1},p);R)$$ which allow us to make sense of stability. Our first result is the following.

\begin{atheorem} Let $p \in \Z$ be a prime and $R$ a field of characteristic zero. The sequence $\{H_i(\Mod(\Sigma_{g,1},p);R) \}$ has generation degree $\leq \begin{cases}
0 & \text{for } i=0 \\
5 & \text{for } i=1 \\
(8)3^{2i-3} & \text{for } i>1. 
\end{cases}$
\label{TheoremMod}
\end{atheorem}

See \autoref{detailedModp} for a version of this theorem which addresses both generation and relation degree. Putman--Sam \cite[Theorem K]{PS} proved that the degree of generation is finite when $R$ is any Noetherian ring, and  \autoref{TheoremMod} quantifies their result when $R$ is a field of characteristic zero. 

A similar story is also true for automorphism groups of free groups. Let $F_n$ denote the free group on $n$ letters. The induced action of $\Aut(F_n)$ on the abelianization $\Z^n$ of $F_n$ gives a surjective map \[\Aut(F_n) \m \GL_n(\Z).\] Reduction mod $p$ gives a surjective map \[\Aut(F_n) \m \GL_n^{\pm}(\Z/p\Z)\] to the subgroup $\GL_n^{\pm}(\Z/p\Z) \subseteq \GL_n(\Z/p\Z)$  of matrices with determinant $\pm 1$. We refer to the kernel of this map as \emph{the level $p$ congruence subgroup of $\Aut(F_n)$} and denote it by $\Aut(F_n,p)$. The natural inclusion $F_n \hookrightarrow F_{n+1}$ gives a $\GL_n^{\pm}(\Z/p\Z)$--equivariant map \[ H_i(\Aut(F_n,p);R) \m H_i(\Aut(F_{n+1},p);R).\] We prove the following theorem. 

\begin{atheorem} \label{thmB} Let $p$ be a prime and $R$ a field of characteristic zero. The sequence $\{H_i(\Aut(F_n,p);R) \}$ has generation degree $\leq \begin{cases}
0 & \text{for } i=0 \\
4 & \text{for } i=1 \\
\left(\frac{13}{2}\right)3^{2i-3} - \frac12 & \text{for } i>1. 
\end{cases}$
\end{atheorem}

See \autoref{detailedAutp} for a version of this theorem which also addresses relation degree. As before, Putman--Sam \cite[Theorem I]{PS} proved that the degree of generation is finite when $R$ is any Noetherian ring. In \autoref{remark3manifold}, we discuss a generalization of \autoref{thmB} which applies to congruence subgroups of automorphism groups of free products of certain fundamental groups of $3$-manifolds ($\Z$, $\Z/2\Z$, $\Z/4\Z$, $\Z/6\Z$, $\pi_1(\Sigma_g)$ etc.). The techniques of Putman--Sam do not apply in this more general context as it is not currently known if the underlying vector spaces are finite dimensional. 

\subsection{Bounding higher syzygies}

To state our main technical tool and to state our homological stability with twisted coefficients theorems, we need the following categories first introduced by Putman--Sam. 

We write $\SI(\bk)$ to denote the category whose objects are finite-rank free symplectic $\bk$--modules and whose morphisms are symplectic embeddings. 

Given a category $\Cat$ and a commutative ring $R$, the term $\Cat$--\emph{module over} $R$ will mean a functor from $\Cat$ to the category of $R$--modules. We denote the category of $\Cat$--modules over $R$ by $\Cat \text{--}\Mod_R$. Given a $\Cat$--module $\cA$ and an object $V$, let $\cA_V$ denote the functor $\cA$ evaluated on $V$. In the case $\Cat=\SI(\bk)$, we write $\cA_n$ to denote $\cA_{\bk^{2n}}$, where we equip $\bk^{2n}$ with a standard sympletic form; see \autoref{SectionCModules}. 


Since the automorphism group of $\bk^{2n}$ in $\SI(\bk)$ is $\Sp_{2n}(\bk)$, the $R$--module $\cA_n$ is naturally a $R[\Sp_{2n}(\bk)]$--module. Symplectic inclusions $\bk^{2n} \hookrightarrow \bk^{2n+2}$ give $R[\Sp_{2n}(\bk)]$--equivariant maps $\cA_n \m \cA_{n+1}$. Thus generation degree is well-defined for $\SI(\bk)$--modules. We will use these constructions to study the homology groups $H_i(\Mod(\Sigma_{g,1},p);R)$, which assemble to form an $\SI(\Z/p\Z)$--module over $R$. 



For $\cA$ an $\SI(\bk)$--module, let \[H_0^{\SI}: \SI(\bk) \text{--}\Mod_R \m \SI(\bk) \text{--} \Mod_R\] be given by the formula
  \[H_0^{\SI}(\cA)_V=\mathrm{coker} \left( \bigoplus_{W \subsetneq V} \cA_W \m \cA_V  \right) \] 
  and let $H_i^{\SI}$ denote the $i$th left derived functor of $H_0^{\SI}$. Details are given in \autoref{SectionCModuleHomology}. 

Vanishing of $H_0^{\SI}$(M) controls the generation degree of $M$ and vanishing of both $H_0^{\SI}$ and $H_1^{\SI}$ control the presentation degree of $M$ (see \autoref{defPresGen} and \autoref{presentLemma}).  Our main technical tool concerning $\SI(\bk)$--modules is the following theorem. 


\begin{atheorem} \label{regSI} Let $\bk$ be a finite field and $R$ a field of characteristic zero. Let $\cA$ be an $\SI(\bk)$--module over $R$ with $H_0^{\SI}(\cA)_n =0 $ for $n > d$ and $H_1^{\SI}(\cA)_n =0 $ for $n > r$. Then for $i \geq 2$ the group $H_i^{\SI}(\cA)_n$ vanishes for $n > 3^{i-1}\max(r,d)$.
\end{atheorem}

The above theorem is analogous to the regularity theorem of Church--Ellenberg \cite[Theorem A]{CE} for $\FI$--modules, although our techniques are different. This theorem shows that we can bound the generation degrees of modules of higher syzygies in terms of the degrees of generators and relations. It serves the same purpose in this paper that the Noetherian theorems serve in Putman--Sam \cite{PS}.

For $\U$ a subgroup of the group of units of $\bk$, let $\GL_n^{\U}(\bk)$ denote the subgroup of matrices with determinant in $\U$. Putman--Sam \cite{PS} introduced a category $\VIC^{\U}(\bk)$ whose automorphism groups are $\GL_n^{\U}(\bk)$, defined in \autoref{defVICU}. The groups $H_i(\Aut(F_n,p))$ assemble to form a $\VIC^\pm(\Z/p\Z)$-module. We prove the following result concerning syzygies of  $\VIC^{\U}(\bk)$--modules.


\begin{atheorem} \label{regVIC} Let $\bk$ be a finite field and $R$ a field of characteristic zero. Let $\cA$ be a $\VIC^{\U}(\bk)$--module over $R$ with $H_0^{\VIC^{\U}}(\cA)_n =0 $ for $n > d$ and $H_1^{\VIC^{\U}}(\cA)_n =0 $ for $n > r$. Then for $i \geq 2$ the group $H_i^{\VIC^{\U}}(\cA)_n $ vanishes for $n >  3^{i-1}\left(\max(r,d)+\frac12\right) -\frac12$.
\end{atheorem}

These theorems imply that when $\bk$ is a finite field and $R$ is a field of characteristic zero, the categories of $\SI(\bk)$-- and $\VIC^{\U}(\bk)$--modules with finite presentation degree are abelian categories; see \autoref{abCat}.


\subsection{Homological stability with twisted coefficients}

Our techniques can also be applied to prove homological stability theorems with twisted coefficients. 

\begin{atheorem} \label{MCGstab} Let $p \in \Z$ be prime. Let $R$ be a field of characteristic zero and let $\cA$ be an $\SI(\Z/p\Z)$--module over $R$ with generation degree $\leq d$ and relation degree $\leq r$. Then an inclusion $\Sigma_{g,1} \hookrightarrow \Sigma_{g+1,1}$ induces an isomorphism \[H_i(\Mod(\Sigma_{g,1});\cA_g) \m H_i(\Mod(\Sigma_{g+1,1});\cA_{g+1}) \] whenever  $$g \geq
\left\{ \begin{array}{ll}
\max(d,r) & \text{for } i=0 \\
\max(9+d+\min(8,d),6+r+\min(5,r), 9+d+\min(5,d)) & \text{for } i=1 \\
\max\big((8)3^{2i-2}+1+d+\min\big((8)3^{2i-2},d\big),(8)3^{2i-3}+1+r+\min\big((8)3^{2i-3},r\big)\big) & \text{for } i>1. \\ 
\end{array} \right.$$
\end{atheorem}

In particular, the conclusion of \autoref{MCGstab} holds for \\[.5em]
$$g \geq \begin{cases}
d+r & \text{for } i=0 \\
17+d+r & \text{for } i=1 \\
1 + (8)3^{2i-2} +2d+2r & \text{for } i>1. \\[.5em]
\end{cases}$$

Similarly, we prove the following stability theorem for automorphism groups of free groups. 

\begin{atheorem} \label{AutFnStab} Let $p \in \Z$ be prime. Let $R$ be a field of characteristic zero and let $\cA$ be a $\VIC^\pm(\Z/p\Z)$--module over $R$ with generation degree $\leq d$ and relation degree $\leq r$. Then the inclusion $F_n \hookrightarrow F_{n+1}$ induces an isomorphism \[H_i(\Aut(F_n);\cA_n) \m H_i(\Aut(F_{n+1});\cA_{n+1}) \] whenever $n$ is at least  
$$ \begin{cases}
\max(d,r) & \text{for } i=0 \\
\max\Big(6+d+\min(6,d), 4+r+\min(4,r) \Big) & \text{for } i=1 \\
\max\Big((\frac{13}{2})3^{2i-2}-\frac12+d +\min\big((\frac{13}{2})3^{2i-2}-\frac12,d\big),(\frac{13}{2})3^{2i-3}-\frac12+r+\min\big(\frac{13}{2})3^{2i-3}-\frac12,r\big)\Big) & \text{for } i>1. 
\end{cases}$$
\end{atheorem}
In particular, the conclusion of \autoref{AutFnStab}  holds for 
$$n \geq \begin{cases}
d+r & \text{for } i=0 \\
12+d+r & \text{for } i=1 \\
(\frac{13}{2})3^{2i-2}-\frac12  +2d+2r & \text{for } i>1. \\[.5em]
\end{cases}$$

These twisted stability theorems are qualitatively different than stability theorems with polynomial coefficients, for example, the coefficients considered in \cite{RWW}. See the discussion before Theorem L in \cite{PS} or Example 1.4 of \cite{MPW} for an exposition of this difference. In fact, the work of Gan--Watterlond \cite{GW} implies that there are no non-constant polynomial coefficient systems in our context.


\subsection{Outline}
In \autoref{AlgebraicResults}, we construct bounded resolutions of $\SI(\bk)$ and $\VIC^{\U}(\bk)$--modules. We use these resolutions in \autoref{Representationstabilityresults} where we prove representation stability for congruence subgroups of mapping class groups and automorphism groups of free groups. We use these representation stability results in \autoref{TwistedStabilityResults} to prove twisted homoloigcal stability theorems for mapping class groups and automorphism groups of free groups.

\subsection{Acknowledgments}

We would like to thank Daniel Bump, Thomas Church, Benson Farb, Rohit Nagpal, Peter Patzt, Andrew Putman, Eric Ramos, and the referee for helpful comments. We thank Rohit Nagpal in particular for identifying an error in an earlier version of a result on the pointwise tensor product of $\VIC(\bk)$--modules.

\section{Algebraic results}
\label{AlgebraicResults}

In this section, we bound the generation degrees of the modules of higher syzygies of $\SI(\bk)$ and $\VIC^{\U}(\bk)$--modules over $R$ that have finite presentation degree. Our main theorems require that $R$ be a field of characteristic zero and that $\bk$ be a finite field. However, many of our intermediate results apply in more generality. 

\subsection{$\Cat$--modules}  \label{SectionCModules}

We begin by defining the categories of interest. All rings are assumed to have unit.

\begin{definition} \label{DefnLinearCategories} Let $R$ and $\bk$ be commutative  rings. Let $\VIC(\bk)$ be the category whose objects are finite-rank free $\bk$--modules, and whose morphisms $U \to V$ are defined to be the set
$$ \Hom_{\VIC}(U, V) = \left\{ \quad f=(T,C) \quad \middle| \quad \begin{array}{l} \text{$T: U \to V$ an injective linear map, } \\  \text{$C$ a specified direct complement of $T(U)$ in $V$.} \end{array} \right\}$$ 
Composition of morphisms is defined by the rule $$ (T, C) \circ (S, D) = (T \circ S, C \oplus f(D)).$$ 
 Similarly $\SI(\bk)$ denotes the category of finite-rank free symplectic $\bk$--modules and injective, isometric embeddings. 
 \end{definition}  
We note that the image of a symplectic embedding $f: V \to W$ has a unique symplectic complement $f(V)^{\perp} \subset W$. 
 
We will use the following generalization of $\VIC$, defined by Putman--Sam \cite[Section 1.2]{PS}.

\begin{definition} Fix a commutative ring $\bk$ and a subgroup $\U \subseteq \bk^{\times}$.  Let 
\[ \GL^{\U}_n(\bk) = \{ A \in \GL_n(\bk) \mid \det A \in \U\}.\]
We write $\VIC^{\U}(\bk)$ to denote the following category. Its objects are finite-rank free $\bk$--modules $V$ such that nonzero objects are assigned a \emph{$\U$--orientation},  a generator of $$\bigwedge^{\text{rank}_{\bk}(V)} V \cong \bk$$ defined up to multiplication by $\U$.   If $V$ and $W$ have the same rank, then $\Hom_{\VIC^{\U}(\bk)}(V,W)$ is the set of linear isomorphisms that respect the designated $\U$--orientations.  If $V$ has strictly smaller rank than $W$, then a morphism $V \to W$ is a complemented injective linear map $f=(T,C)$, for which we assign to $C$ the unique $\U$--orientation such that $$T(V) \oplus C \cong W \qquad \text{as oriented $\bk$--modules.}$$ Here $T(V)$ is equipped with the orientation induced by the $\U$--orientation on $V$.
\label{defVICU}
 \end{definition}
In particular,  
$$\End_{\VIC^{\U}(\bk)}(\bk^n) \cong \GL_n^{\U}(\bk), $$
but if $V$ has strictly smaller rank than $W$, then 
$$\Hom_{\VIC^{\U}(\bk)}(V,W) \cong  \Hom_{\VIC(\bk)}(V,W). $$
 When $\U = \{1, -1\}$, we write $\VIC^{\pm}(\bk)$ for $\VIC^{\U}(\bk)$. Note that when $\U=\bk^{\times}$, the category $\VIC^{\U}$ is isomorphic to $\VIC$. 
 
   For convenience, we will often work with a skeleton of the category $\VIC(\bk)$ or $\VIC^{\U}(\bk)$, the full subcategory with objects $\bk^d$, $d\geq 0$. Given these choices of bases for our objects, when convenient we can represent our morphisms $(T,C): \bk^d \to \bk^n$ by an equivalence class of $(n \times n)$ matrices in $\GL^{\U}_n(\bk)$ where the first $d$ columns are the matrix representative for $T$, and the final $(n-d)$ columns span $C$.   Similarly, we may choose a skeleton of $\SI(\bk)$ of symplectic vector spaces $\bk^{2d}$  with symplectic form 
$$\Omega_d = \begin{bmatrix} 0 & 1 & && \\ -1 & 0 & &0&  \\ & & \ddots \\ & 0& & 0 & 1 \\ & & & -1 & 0 \end{bmatrix}.$$
A morphism $\bk^{2d} \to \bk^{2n}$  is given by a $(2n \times 2d)$ matrix $A$ that satifies $A^T \Omega_n A  = \Omega_d$. 



\begin{remark} \label{RemarkStabilizer} Consider the action of $\End_{\VIC^{\U}(\bk)}(\bk^n) \cong \GL^{\U}_n(\bk)$ on the morphisms $\Hom_{\VIC^{\U}(\bk)}(\bk^d, \bk^n)$. A morphism $f=(T,C)$ has stabilizer $\GL^{\U}(C) \cong \GL^{\U}_{n-d}(\bk)$ in $\GL^{\U}_n(\bk)$. Similarly, a morphism $f \in \Hom_{\SI(\bk)}(\bk^{2d}, \bk^{2n})$ has stabilizer $\Sp( f(\bk^{2d})^{\perp}) \cong \Sp_{2n-2d}(\bk)$ in $\Sp_{2n}(\bk)$, where again $ f(\bk^{2d})^{\perp}$ denotes the symplectic complement of $f(\bk^{2d}) \subseteq \bk^{2n}$. 
\end{remark}

Throughout the paper we will let $\Cat$ generically refer to the category $\SI(\bk)$ or $\VIC^{\U}(\bk)$, and denote the endomorphisms $\End_{\VIC^{\U}(\bk)}(\bk^n) \cong \GL^{\U}_n(\bk)$ or $\End_{\SI(\bk)}(\bk^{2n}) \cong \Sp_{2n}(\bk)$ generically  by $\G_n$. We stress that for the category $\SI(\bk)$, these indices $n$ are half the rank of the corresponding symplectic $\bk$--module $\bk^{2n}$.

\begin{definition} We write $\CB$ to denote the subcategory of $\Cat$ with the same objects as $\Cat$, whose morphisms are all isomorphisms of $\Cat$. A $\CB$--module $W$ is therefore a sequence $W=\{W_n\}$ of $\G_n$--representations, and we define the \emph{support} of a $\CB$--module to be the set $\{ n \in \Z_{\geq 0} \; | \; W_n \neq 0 \}$. 
\end{definition}


\begin{definition} Let $\cM(d)$ denote the representable $\VIC^{\U}(\bk)$--module 
$$ \bk^n \longmapsto R\left[ \Hom_{\VIC^{\U}(\bk)}(\bk^d, \bk^n) \right] $$
or the representable $\SI(\bk)$--module  
$$ \bk^{2n} \longmapsto R\left[ \Hom_{\SI(\bk)}(\bk^{2d}, \bk^{2n}) \right].$$
\end{definition}
In both cases such a morphism has stabilizer $\G_{n-d}$ by  \autoref{RemarkStabilizer}, and so there are isomorphisms of $\G_n$--representations  $$\cM(d)_n \; \cong\; R \left[ \G_n/\G_{n-d} \right]  \; \cong\; \Ind_{ \G_{n-d}}^{\G_n} R \; \cong\; \Ind_{\G_d \times \G_{n-d}}^{\G_n} R[\G_d] \boxtimes R  $$  where $R$ denotes the trivial $\G_{n-d}$--representation. 

We sometimes write $\cM^{\VIC}(d)$, $\cM^{\VIC^{\U}}(d)$ or $\cM^{\SI}(d)$ for $\cM(d)$ when we wish to specialize to a particular category $\Cat=\VIC(\bk)$, $\Cat=\VIC^{\U}(\bk)$, or $\Cat=\SI(\bk)$.


\begin{definition} \label{DefMLeftAdjoint}   We define the functor 
\begin{align*} 
\cM: \text{$\CB$--Mod} &\longrightarrow  \text{$\Cat$--Mod} 
\end{align*}
to be the left adjoint to the forgetful functor
\begin{align*} \cF: \text{$\Cat$--Mod} &\longrightarrow  \text{$\CB$--Mod}  \\
\cA &\longmapsto  \{\cA_n \}
\end{align*}

Concretely, given a $\G_d$--representation $W$ (viewed as a $\CB$--module supported in degree $d$), the $\Cat$--module $\cM(W)$ satisfies
 $$\cM(W) = \cM(d) \otimes_{R [\G_d]} W. $$ 
As a $\G_n$--representation, 
 $$\cM(W)_n \cong \left\{ \begin{array}{ll} 0 & n<d \\
 \Ind_{\G_d \times \G_{n-d}}^{\G_n} W \boxtimes R & n \geq d. \end{array}\right.$$
Given a general $\CB$--module $W= \{ W_n  \}$, the $\Cat$--module $\cM(W)$ is given by the formula
\begin{align*} 
\cM: \text{$\CB$--Mod} &\longrightarrow  \text{$\Cat$--Mod}\\
 \{ W_n  \} &\longmapsto \bigoplus_{m \geq 0} \cM(W_m)
\end{align*}
These formulas follow as in Church--Ellenberg--Farb \cite[Definition 2.2.2 and Equation (4)]{CEF}. 
Following the terminology of Nagpal--Sam--Snowden \cite{NSS}, we call $\Cat$--modules of this form \emph{induced} $\Cat$--modules.  Again we sometimes write $\cM^{\VIC}(W)$, $\cM^{\VIC^{\U}}(W)$, or $\cM^{\SI}(W)$ for $\cM(W)$ when $\Cat=\VIC(\bk)$, $\VIC^{\U}(\bk)$,  or $\SI(\bk)$. \end{definition}

\begin{proposition} \label{CorMWProjective} For any projective $\CB$--module $W$, the $\Cat$--module $\cM(W)$ is projective. In particular, if $\bk$ is a finite commutative ring and $R$ a field of characteristic zero, then $\cM(W)$ is projective for all $\CB$--modules $W$.  
\end{proposition}

\begin{proof} Since $\cM$ is the left adjoint of the exact forgetful functor, it preserves projectives; see Weibel \cite[Proposition 2.3.10]{Weibel}. When the algebras $R[\G_n]$ are semi-simple then all $\CB$--modules are projective. 
\end{proof}

From the formula for $\cM(W)$, and because induction of group representations is an exact operation, we deduce the following. 

\begin{proposition} \label{MExact} The functor $\cM: \CB\text{--Mod} \to \Cat\text{--Mod}$ is exact. 
\end{proposition}

We now introduce the concepts of generation, relation, and presentation degree. 

\begin{definition}  \label{defPresGen}
A $\Cat$--module $\cA$ is \emph{generated in degree  $\leq d$} if $\cA$ can be expressed as a quotient of a $\Cat$--module of the form 
$$ \cM(W) \twoheadrightarrow \cA$$ 
for some $\CB$--module $W$ supported in degrees $\leq d$.  We say that $\cA$ is \emph{related in degree $\leq r$} if $\cA$ can be expressed as a quotient as above whose kernel is generated in degree $\leq r$. If $\cA$ is generated in degree $\leq d$ and related in degree $\leq r$, we say it has \emph{presentation degree $\leq \max(d,r)$.} 
\end{definition}

\begin{proposition} \label{RemarkIndSurjects} Let $\cA$ be a $\Cat$--module. The following statements are all equivalent to the condition that $\cA$ is generated in degree $\leq d$. 
\begin{enumerate}[label=(\alph*)]
\item \label{a} $\cA$ is a quotient of an induced $\Cat$--module
$ \cM(W) \twoheadrightarrow \cA$ 
with $W$ supported in degrees $\leq d$. 
\item \label{b} For all $n \geq d$, the $\G_{n+1}$--representation $\cA_{n+1}$ is generated by the image of $\cA_n$ in $\cA_{n+1}$ under any map induced by a $\Cat$ morphism.
\item \label{c} For all $n \geq d$, the $\Cat$ morphisms induce surjections $\Ind_{\G_n}^{\G_{n+1}} \cA_{n}  \twoheadrightarrow \cA_{n+1}$. 
\item \label{d} The subset $\{ \cA_n \}_{n=0}^d$ of $\cA$ is not contained in any proper $\Cat$--submodule of $\cA$.
\item \label{e} The inclusion of $\CB$--modules $\{ \cA_n \}_{n=0}^d \hookrightarrow \{\cA_n\}$ induces a surjective map of $\Cat$--modules  $$\cM\left( \{ \cA_n \}_{n=0}^d \right)  \twoheadrightarrow \cA.$$ 
\end{enumerate} \end{proposition} 
\begin{proof} We can verify directly that if $W$ is supported in degrees $\leq d$ then $\cM(W)$ satisfies \ref{b}, and hence its $\Cat$--module quotients do. Thus \ref{a} implies \ref{b}. Parts \ref{b} and \ref{c} are equivalent by definition of induction. It is straightforward to conclude \ref{d}  from \ref{b}. Part \ref{d}  implies that any map of $\Cat$--modules to $\cA$ that is surjective in the first $d$ degrees must surject in all degrees, and so implies \ref{e}. Part \ref{a} is immediate from part \ref{e}. \end{proof}

\begin{remark} We note that the induced $\Cat$--module $\cM(W)$ is generated in degree $\leq d$ if and only if $W$ is supported in degree $\leq d$. 
\end{remark}


\begin{proposition} \label{InducedVsRepresentable}  Let $\bk$ be a finite field and let $R$ be a field of characteristic zero. Any induced $\Cat$--module $\cM(W)$ can be realized as both a $\Cat$--module quotient, and a $\Cat$--submodule, of $\Cat$--modules of the form $$\bigoplus_{m=0}^{\infty} \cM(m)^{\oplus c_m}$$ for some (possibly infinite) multiplicities $c_m$. If $\cM(W)$ is generated in degree $\leq d$, then we can realize it as a quotient or a submodule of $\Cat$--modules of the form $$\bigoplus_{m=0}^{d} \cM(m)^{\oplus c_m}.$$
More generally, if $\cA$ is any $\Cat$--module generated in degree $\leq d$, then we can realize $\cA$ as a quotient of a $\Cat$--module of the form $$\bigoplus_{m=0}^{d} \cM(m)^{\oplus c_m}.$$
\end{proposition} 
Notably, the following constructions are valid even if we allow the $R[\G_n]$--representations $W_n$ to be infinite-dimensional.

\begin{proof}[Proof of \autoref{InducedVsRepresentable}] Observe that we can construct a $\CB$--module $\{ R[\G_n]^{\oplus c_n} \}$ so as to obtain a map of $\CB$--modules $$ R[\G_n]^{\oplus c_n} \to W_n $$ that surjects in each degree $n$. If $W$ is supported in degree $\leq d$ we may take $c_n=0$ for $n>d$. Applying the functor $\cM$ we obtain a map of $\Cat$--modules,  $$ \cM\left( \{ R[\G_m]^{\oplus c_m} \} \right) = \bigoplus_m \cM\left( m \right)^{\oplus c_m}  \longrightarrow \cM(W),$$
and by \autoref{MExact} this map surjects. 

Moreover, since the algebras $R[\G_n]$ are semi-simple by assumption, the maps $ R[\G_n]^{\oplus c_n} \to W_n $ split to give an injective map of $\CB$--modules $W \to  \{ R[\G_n]^{\oplus c_n} \}$. Again the induced map  $$\cM(W) \longrightarrow   \bigoplus_m \cM\left( m \right)^{\oplus c_m}$$ is injective by \autoref{MExact}. 

Finally, if $\cA$ is any $\Cat$--module generated in degree $\leq d$, then by definition of generation degree we can realize $\cA$ as a quotient $\cM(W) \to \cA$ with $W$ supported in degree $\leq d$. Then we may compose this map with the surjection constructed above to obtain the desired surjective map
\[ \bigoplus_{m=0}^d \cM\left( m \right)^{\oplus c_m}  \longrightarrow \cM(W) \longrightarrow  \cA. \qedhere \]
\end{proof}

\subsection{Weight and stability degree}

In this subsection, we will introduce concepts of \emph{weight} and \emph{stability degree} for $\Cat$--modules, closely analogous to the concepts of the same name used by Church, Ellenberg, and Farb \cite{CEF} in the study of $\FI$--modules. These $\Cat$--module invariants will be our main tool for bounding the generation degrees of the terms in resolutions of $\Cat$--modules. 

\begin{definition} \label{DefnWeight} 
 A $\Cat$--module $\cA$ has \emph{weight $\leq d$} if for each $n$, the $\G_n$--representation $\cA_n$ is a subquotient of a representation of the form $\bigoplus_{m\leq d} \cM(m)_n^{\oplus c_m}$ for some (possibly infinite) coefficients $c_m$.
\end{definition}

\begin{remark} \label{RemarkWeightSubquotients}
It follows from the definition that if $\cA$ is a $\Cat$--module of weight $\leq d$, then any subquotient of $\cA$ has weight $\leq d$. 
\end{remark}

\begin{remark} \label{RemarkWeightGeneration}
By \autoref{InducedVsRepresentable}, any $\Cat$--module $\cA$ generated in degree $\leq d$ must be a quotient of the form in \autoref{DefnWeight}, and so $\cA$ has weight $\leq d$. 
\end{remark}


\begin{lemma} \label{LemmaCoinvariantsVanish} Let $\bk$ be a finite commutative ring, and let $R$ be a field of characteristic zero. Suppose that $\cA$ is a $\Cat$--module over $R$ of weight $\leq d$, and that $C_n$ is any subquotient of the $\G_n$--representation $\cA_n$. Then $C_n = 0$ if and only if $(C_n)_{\G_{n-d}} = 0$.
\end{lemma}

An analogous statement for $\FI$--modules was proved by Church--Ellenberg--Farb \cite[Lemma 3.2.7(iv)]{CEF}. Their proof uses combinatorial properties of the branching rules for induction of symmetric group representations. The following proof instead uses Frobenius reciprocity. 

\begin{proof}[Proof of  \autoref{LemmaCoinvariantsVanish}]

If $C_n=0$, then its coinvariants must vanish. So suppose that  $(C_n)_{\G_{n-d}} = 0$. Note that if $m \leq d$, then $(C_n)_{\G_{n-m}}$ is a quotient of $(C_n)_{\G_{n-d}}$ and therefore also vanishes. 
To verify that $C_{n}$ vanishes, it is enough to show that 
$$\Hom_{R[\G_n]}( U, C_n ) = 0 \quad \text{ for all $\G_n$--representations $U$.} $$ 
By the definition of weight, and because we are working with finite groups over characteristic zero, any irreducible subrepresentation of $C_n$ must be contained in a $\G_n$--representation $U$ of the form 
$$U = \Ind_{\G_m \times \G_{n-m}}^{\G_n} R[\G_m] \boxtimes R \qquad \text{ with $m \leq d$} $$
so it suffices to check that 
$\Hom_{R[\G_n]}( U, C_n ) = 0 $  in this case.
Using Frobenius reciprocity (or the tensor-Hom adjunction), we find
\begin{align*}
\Hom_{R[\G_n]} &  \left( \mathrm{Ind}^{\G_n}_{\G_m \times \G_{n-m}} R[\G_m] \boxtimes R, \;  C_n \right) \\
 &= \Hom_{R[\G_m \times \G_{n-m}]} \left( R[\G_m] \boxtimes R , \; \mathrm{Res}^{\G_n}_{\G_m \times \G_{n-m}}  C_n \right) \\
&=  \Hom_{R[\G_m]}  \left( R[\G_m],  \; (C_n)_{\G_{n-m}} \right) \\ 
& = 0  \end{align*} 
as claimed. 
\end{proof}


\begin{definition} A $\Cat$--module $\cA$ has \emph{stability degree $\leq s$} if for each $a \geq 0$, the induced map 
$$ (\cA_n)_{\G_{n-a}} \longrightarrow (\cA_{n+1})_{\G_{n+1-a}}  $$ is an isomorphism for all $n \geq s+a$. We further say that $\cA$ has \emph{injectivity degree $\leq s$} if these maps are injective for $n \geq s+a$, and \emph{surjectivity degree $\leq s$} if these maps are surjective for $n \geq s+a$. We use the notation $\mathrm{InjDeg}(\cA) \leq s$ (respectively, $\mathrm{SurjDeg}(\cA) \leq s$) to indicate that $\cA$ has injectivity degree (respectively, surjectivity degree) $\leq s$. 
\end{definition} 

\begin{proposition} \label{StabDegQuotientsSubmodules}  Let $\bk$ be a finite commutative ring, let $R$ be a field of characteristic zero, and let $\cB$ be a  $\Cat$--module over $R$. If $\cB$ has surjectivity degree $\leq s$, then so does any quotient of $\cB$.  If $\cB$ has injectivity degree $\leq t$, then so does any submodule of $\cB$. 
\end{proposition} 
\begin{proof} Suppose that $\cA$ is a submodule of $\cB$ and that $\cC$ is a quotient. Since the operation of taking coinvariants by a finite group is exact over characteristic zero, we obtain the following commutative diagrams. For $n \geq a+t$, the diagram
\begin{center} \begin{tikzpicture}
  \matrix (m) [matrix of math nodes,row sep=3em,column sep=4em,minimum width=2em]
  {
     (\cA_n)_{\G_{n-a}} & (\cA_{n+1})_{\G_{n+1-a}} \\
     (\cB_n)_{\G_{n-a}} & (\cB_{n+1})_{\G_{n+1-a}}  \\};
   \path[-stealth] (m-1-1) edge (m-1-2);
 \path[right hook->]    (m-1-1) edge(m-2-1)  
    (m-2-1) edge (m-2-2)
    (m-1-2) edge  (m-2-2);
\end{tikzpicture} \end{center}
implies that the map $(\cA_n)_{\G_{n-a}}  \longrightarrow (\cA_{n+1})_{\G_{n+1-a}}$ injects.

 For $n \geq a+s$,  the diagram

\begin{center} \begin{tikzpicture}
  \matrix (m) [matrix of math nodes,row sep=3em,column sep=4em,minimum width=2em]
  {
    (\cB_n)_{\G_{n-a}} & (\cB_{n+1})_{\G_{n+1-a}}\\
     (\cC_n)_{\G_{n-a}} & (\cC_{n+1})_{\G_{n+1-a}}  \\};
   \path[-stealth]    (m-2-1) edge  (m-2-2);
 \path[->>]   (m-1-1) edge (m-1-2)
 (m-1-1) edge(m-2-1)  
    (m-1-2) edge (m-2-2);
\end{tikzpicture} \end{center}
shows that the map $ (\cC_n)_{\G_{n-a}} \longrightarrow (\cC_{n+1})_{\G_{n+1-a}}$ is surjective. 
\end{proof}

\begin{proposition} \label{StabDegKernelsCokernels}
Let $\bk$ be a finite commutative ring, let $R$ be a field of characteristic zero, and let $f : \cA \to \cB$ be a map of $\Cat$--modules over $R$. Then
\begin{align*} \mathrm{InjDeg}({\ker f}) &\leq \mathrm{InjDeg}(\cA) &&& \mathrm{SurjDeg}({\ker f}) &\leq \max\Big( \mathrm{SurjDeg}(\cA), \mathrm{InjDeg}(\cB) \Big) &\\ 
\mathrm{InjDeg}({\coker f}) &\leq \max\Big( \mathrm{SurjDeg}(\cA), \mathrm{InjDeg}(\cB) \Big) &&& \mathrm{SurjDeg}({\coker f})& \leq \mathrm{SurjDeg}(\cB). &
\end{align*}
\end{proposition}

\begin{proof} 
The results $\mathrm{InjDeg}({\ker f}) \leq \mathrm{InjDeg}(\cA)$ and $\mathrm{SurjDeg}({\coker f}) \leq \mathrm{SurjDeg}(\cB)$ follow from \autoref{StabDegQuotientsSubmodules}. 
 Since taking coinvariants is exact over $R$, for $n \geq a+ \max\Big( \mathrm{SurjDeg}(\cA), \mathrm{InjDeg}(\cB) \Big) $ we obtain the following commutative diagram with exact columns

\begin{center} \begin{tikzpicture}
  \matrix (m) [matrix of math nodes,row sep=3em,column sep=4em,minimum width=2em]
  { (\ker f_n)_{\G_{n-a}} & (\ker f_{n+1})_{\G_{n+1-a}}\\
    (\cA_n)_{\G_{n-a}} & (\cA_{n+1})_{\G_{n+1-a}}\\
     (\cB_n)_{\G_{n-a}} & (\cB_{n+1})_{\G_{n+1-a}}  \\};
 \path[->>]     (m-2-1) edge  (m-2-2);
      \path[-stealth]    (m-2-1) edge node[right] {$f_*$} (m-3-1)  
    (m-2-2) edge node[left] {$f_*$} (m-3-2)
 (m-1-1) edge (m-1-2);
 \path[right hook->]  (m-1-1) edge(m-2-1)  
    (m-1-2) edge (m-2-2)
 (m-3-1) edge (m-3-2);
\end{tikzpicture} \end{center}
A routine diagram chase demonstrates that the map $ (\ker f_n)_{\G_{n-a}} \longrightarrow (\ker f_{n+1})_{\G_{n+1-a}}$ surjects, as desired. 
We also obtain, for $n \geq a+ \max\Big( \mathrm{SurjDeg}(\cA), \mathrm{InjDeg}(\cB) \Big) $ the commutative diagram with exact columns

\begin{center} \begin{tikzpicture}
  \matrix (m) [matrix of math nodes,row sep=3em,column sep=4em,minimum width=2em]
  {  (\cA_n)_{\G_{n-a}} & (\cA_{n+1})_{\G_{n+1-a}}\\
     (\cB_n)_{\G_{n-a}} & (\cB_{n+1})_{\G_{n+1-a}}  \\
   (\coker f_n)_{\G_{n-a}} & (\coker f_{n+1})_{\G_{n+1-a}}\\};
 \path[->>]     (m-1-1) edge  (m-1-2)
 (m-2-2) edge (m-3-2)
 (m-2-1) edge  (m-3-1) ;
 \path[-stealth]    (m-1-1) edge node[right] {$f_*$}(m-2-1)  
 (m-1-2) edge node[left] {$f_*$}  (m-2-2)
 (m-3-1) edge (m-3-2);
 \path[right hook->]   (m-2-1) edge (m-2-2);
\end{tikzpicture} \end{center}
We can verify that the kernel of the map $  (\coker f_n)_{\G_{n-a}} \longrightarrow (\coker f_{n+1})_{\G_{n+1-a}}$ vanishes, which concludes the proof. \end{proof}

Patzt proved the following result on the stability degree of representable $\Cat$--modules. We remark that, although he only states the results for $\bk=\Q$, his proof only uses the assumption that $\bk$ is a field.

\begin{proposition}[{Patzt \cite[Proposition 3.11]{PatztFiltrations}}] \label{TheoremStabDegVICM(d)} Let $R$ be ring and $\bk$ a field. Let $\Cat$ be the category $\VIC(\bk)$ or $\SI(\bk)$. Then the representable  $\C$--module $\cM(d)$ over $R$ has injectivity degree $\leq 0$, and surjectivity degree $\leq 2d$.
\end{proposition}

We now explain how to leverage this result to prove an analogous statement for $\VIC^{\U}(\bk)$.

\begin{proposition} \label{VICUcoset} Let $R$ be ring, $\bk$ a field, and $\U$ a subgroup of $\bk^\times$. Then the representable  $\VIC^{\U}(\bk)$--module $\cM(d)$ over $R$ has injectivity degree $\leq 2d+1$, and surjectivity degree $\leq 2d$.
\end{proposition}

\begin{proof}
Fix $a \geq 0$. Let $\I_m$ denote the $(m \times m)$ identity matrix. By definition the $\GL^{\U}_n(\bk)$--representation $\cM(d)_n$ is a permutation representation with $R$--basis the set of cosets $\GL^{\U}_n(\bk) / \GL^{\U}_{n-d}(\bk)$. It follows that its coinvariants $\left(\cM(d)_n\right)_{\GL^{\U}_{n-a}(\bk)}$ has $R$--basis the set of double cosets $$\GL^{\U}_{n-a}(\bk) \backslash \GL^{\U}_n(\bk) / \GL^{\U}_{n-d}(\bk).$$ Concretely, this is the set of $(n \times n)$ matrices $B$ with determinants in $\U$ defined up to the action of 
$$ \GL^{\U}_{n-a}(\bk) \cong \left\{  \left[\begin{array}{ccc|ccc} &  \I_a & && 0 &  \\[.5em] \hline &&&&&\\[-.5em]  & 0 &&& \bigstar &  \end{array} \right] \right\} $$
on the left -- acting by row operations on the bottom $(n-a)$ rows of $B$  -- and the action of 
$$ \GL^{\U}_{n-d}(\bk) \cong  \left\{ \left[\begin{array}{ccc|ccc} &  \I_d & && 0 &  \\[.5em] \hline &&&&&\\[-.5em]  & 0 &&& \bigstar &  \end{array} \right] \right\} $$ on the right, acting by column operations on the rightmost $(n-d)$ columns of $B$.
The map  $$(\cM(d)_n)_{\GL^{\U}_{n-a}(\bk)}\to (\cM(d)_{n+1})_{\GL^{\U}_{n+1-a}(\bk)}$$ defining stability degree is induced by the map 
\begin{align*}
\GL^{\U}_n(\bk) \longrightarrow& \GL^{\U}_{n+1}(\bk) \\ 
C \longmapsto&  \begin{bmatrix}   & 0 \\ C & \vdots \\ & 0 \\   0  \cdots 0 & 1\end{bmatrix}  .
\end{align*}
We will first establish the bound on surjectivity degree for $\cM(d)$, by proving that the map
$$ \GL^{\U}_{n-a}(\bk) \backslash \GL^{\U}_n(\bk) / \GL^{\U}_{n-d}(\bk) \longrightarrow \GL^{\U}_{n+1-a}(\bk) \backslash \GL^{\U}_{n+1}(\bk) / \GL^{\U}_{n+1-d}(\bk) $$ 
 surjects for $n \geq 2d+a$. When $d=0$, the domain and codomain are both singleton sets and the result is immediate, so we may assume $d>0$. Let $B$ be any matrix in $\GL^{\U}_{n+1}(\bk)$. Patzt proved that $\cM^{\VIC}(d)$ has surjectivity degree $\leq 2d$; specifically, he found matrices $A \in \GL_{n+1-a}(\bk)$, $D \in \GL_{n+1-d}(\bk)$, and $C \in \GL_n(\bk)$ so that 
 $$ \begin{bmatrix}  \I_a &  \\  & A\end{bmatrix} B \begin{bmatrix}  \I_d &  \\  & D \end{bmatrix} = \begin{bmatrix}   & 0 \\ C & \vdots \\ & 0     \\  0  \cdots 0 & 1\end{bmatrix}.$$
Our goal is to modify $A$, $D$, and $C$ so that they have determinants in $\U$. Observe that
 {\scriptsize \begin{align*}   
&  \left(  \begin{bmatrix}  1 &&&&  \\ &\ddots&&& \\ &&1&& \\  &&& \det(C)^{-1} & \\ &&&& \det(D)\end{bmatrix} \begin{bmatrix}  \I_a &  \\  & A\end{bmatrix} \right)
B 
\left(  \begin{bmatrix}  \I_d &  \\  & D \end{bmatrix}  \begin{bmatrix}  1 &&&  \\ &\ddots& \\ &&1 \\  &&& \det(D)^{-1} \\  \end{bmatrix}   \right ) \\
& \qquad =  \begin{bmatrix}  1 &&&&  \\ &\ddots&&& \\ &&1&& \\  &&& \det(C)^{-1} & \\ &&&& \det(D) \end{bmatrix}   \begin{bmatrix}   & 0 \\ C & \vdots \\ & 0 \\    \\  0  \cdots 0 & 1\end{bmatrix}  \begin{bmatrix}  1 &&&  \\ &\ddots& \\ &&1 \\  &&& \det(D)^{-1}\\  \end{bmatrix}   \\
& \qquad =  \begin{bmatrix}   & 0 \\ C' & \vdots \\ & 0 \\    \\  0  \cdots 0 & 1\end{bmatrix}  
 \end{align*} }
\hspace{-6pt} where $C'$ is obtained by scaling the bottom row of $C$ by $\det(C)^{-1}$. Thus $C'$ has determinant $1$, and the matrix on the right-hand side of the equation is in the image of $\GL_n^{\U}(\bk)$.  
Since $$n+1-d \geq 1+d+a  \geq 1$$ by assumption, the matrix 
\begin{align*} {\scriptsize 
\left(  \begin{bmatrix}  \I_d &  \\  & D \end{bmatrix}  \begin{bmatrix}  1 &&&  \\ &\ddots& \\ &&1 \\  &&& \det(D)^{-1} \\  \end{bmatrix}   \right ) 
} \end{align*} 
is contained in $\GL_{n+1-d}(\bk)$, and moreover has determinant $1$. Since $$n+1-a \geq 2d+1 \geq 2$$  the matrix 
\begin{align*}  {\scriptsize 
  \left(  \begin{bmatrix}  1 &&&&  \\ &\ddots&&& \\ &&1&& \\  &&& \det(C)^{-1} & \\ &&&& \det(D) \end{bmatrix} \begin{bmatrix}  \I_a &  \\  & A\end{bmatrix} \right)
} \end{align*} 
is contained in $\GL_{n+1-a}(\bk)$, and must have determinant $\det(B)^{-1} \in \U$.  This concludes the proof of the bound on surjectivity degree. 

We next prove the bound on injectivity degree. Let $n \geq 2d+a+1$, and  we will show that the map on double cosets 
$$ \GL^{\U}_{n-a}(\bk) \backslash \GL^{\U}_n(\bk) / \GL^{\U}_{n-d}(\bk) \longrightarrow \GL^{\U}_{n+1-a}(\bk) \backslash \GL^{\U}_{n+1}(\bk) / \GL^{\U}_{n+1-d}(\bk) $$ 
is injective. 
Suppose that $[B]$ and $[C]$ are double cosets in $\GL^{\U}_n(\bk)$ that map to the same double coset in $\GL^{\U}_{n+1}(\bk)$. Since the map on double cosets is surjective for $n \geq 2d+a$, we may assume without loss of generality that $[C]$ is represented by a matrix of the form
$$  {\scriptsize C= \begin{bmatrix} & 0 \\ C' & \vdots \\ & 0 \\  0  \cdots 0 & 1\end{bmatrix} }  \qquad \in \GL_n^{\U}(\bk).$$
 Patzt proved that $\cM^{\VIC}(d)$ has injectivity degree $0$ by exhibiting matrices $A \in \GL_{n-a}(\bk)$ and $D \in \GL_{n-d}(\bk)$ so that 
 $$ \begin{bmatrix}  \I_a &  \\  & A\end{bmatrix} B \begin{bmatrix}  \I_d &  \\  & D \end{bmatrix} = C.$$
Now observe that 
{\scriptsize \begin{align*}  
& \left(  \begin{bmatrix}  1 &&&  \\ &\ddots& \\ &&1 \\  &&& \det(D)\\ \end{bmatrix} \begin{bmatrix}  \I_a &  \\  & A\end{bmatrix}  \right) B \left( \begin{bmatrix}  \I_d &  \\  & D \end{bmatrix}    \begin{bmatrix}  1 &&&  \\ &\ddots& \\ &&1 \\  &&& \det(D)^{-1}\\  \end{bmatrix} \right)    \\
& \qquad \qquad  = \begin{bmatrix}  1 &&&  \\ &\ddots& \\ &&1 \\  &&& \det(D)\\  \end{bmatrix}  \begin{bmatrix}   & 0 \\ C' & \vdots \\ & 0 \\    0  \cdots 0 & 1\end{bmatrix}  \begin{bmatrix}  1 &&&  \\ &\ddots& \\ &&1 \\  &&& \det(D)^{-1}\\  \end{bmatrix}   \\
& \qquad \qquad = \begin{bmatrix}   & 0 \\ C' & \vdots \\ & 0 \\    0  \cdots 0 & 1\end{bmatrix}
\end{align*}} 
where 
\begin{align*}
{\scriptsize  \left( \begin{bmatrix}  \I_d &  \\  & D \end{bmatrix}    \begin{bmatrix}  1 &&&  \\ &\ddots& \\ &&1 \\  &&& \det(D)^{-1}\\  \end{bmatrix} \right) }  \in \GL_{n-d}(\bk) \quad \text{has determinant 1}, 
\end{align*} 
and 
 \begin{align*}
{\scriptsize \left(  \begin{bmatrix}  1 &&&  \\ &\ddots& \\ &&1 \\  &&& \det(D)\\ \end{bmatrix} \begin{bmatrix}  \I_a &  \\  & A\end{bmatrix}  \right) }  \in \GL_{n-a}(\bk) \quad \text{has determinant }\det(C')\det(B)^{-1} \in \U. 
\end{align*} 
Thus $[C]$ and $[B]$ are the same double coset in $\GL^{\U}_{n-a}(\bk) \backslash \GL^{\U}_n(\bk) / \GL^{\U}_{n-d}(\bk)$, and we conclude the bound on injectivity degree.
\end{proof}

 From \autoref{TheoremStabDegVICM(d)} and \autoref{VICUcoset} we will deduce the following results for general $\Cat$--modules. 

\begin{proposition} \label{StabDegSharp} Let $R$ be a field of characteristic zero and $\bk$ a finite field. 
 Any $\Cat$--module $\cA$ over $R$ generated in degree $\leq d$ has surjectivity degree $\leq 2d$.  If $\cA$ is an induced module over $\VIC(\bk)$ or $\SI(\bk)$, then $\cA$ has injectivity degree $\leq 0$. Induced $\VIC^{\U}$--modules generated in degree $\leq d$ have injectivity degree $\leq 2d+1$. 
\end{proposition}

\begin{proof} Since by \autoref{InducedVsRepresentable} any $\Cat$--module generated in degree $\leq d$ can be realized as a quotient of a direct sum of $\Cat$--modules $\cM(m)$ with $m \leq d$, the result follows from  \autoref{TheoremStabDegVICM(d)}, \autoref{VICUcoset},  and   \autoref{StabDegQuotientsSubmodules}. If  $\cA$ is the $\Cat$--module  $\cM(W)$ for some $\CB$--module $W=\{W_n\}$, then by \autoref{InducedVsRepresentable}  we can realize $\cA$ as a submodule of a direct sum of representable $\Cat$--modules $\cM(m)$, and the result again follows from  \autoref{TheoremStabDegVICM(d)}, \autoref{VICUcoset},  and \autoref{StabDegQuotientsSubmodules}. 
\end{proof}

The following result shows that the stability degree of a general $\Cat$--module is controlled by its presentation degree. 

\begin{proposition} \label{boundingstabdeg} Let $\bk$ be a finite field, and let $R$ be a field of characteristic zero. Let $\Cat$ be $\VIC(\bk)$ or $\SI(\bk)$, and suppose that $\cA$ is a $\Cat$--module over $R$ with generation degree $\leq d$ and relation degree $\leq r$. Then $\cA$ has stability degree $\leq \max(2r,2d)$. If  $\Cat$ is $\VIC^{\U}(\bk)$, and $\cA$ is a $\Cat$--module over $R$ with generation degree $\leq d$ and relation degree $\leq r$, then $\cA$ has stability degree $\leq \max(2r,2d+1)$.
\end{proposition}

\begin{proof}
By assumption, there exists a  partial resolution of $\cA$ by  induced $\Cat$--modules
$$ \cM^1 \longrightarrow \cM^0 \longrightarrow \cA$$
with $\cM^1$ generated in degree $\leq r$ and $\cM^0$ generated in degree $\leq d$.  When $\Cat$ is $\VIC(\bk)$ or $\SI(\bk)$,  $\cM^1$ and $\cM^0$ have injectivity degree $\leq 0$ and surjectivities degrees $\leq 2r$ and $\leq 2d$, respectively, by \autoref{StabDegSharp}. When $\Cat$ is $\VIC^{\U}(\bk)$, then by \autoref{StabDegSharp}, $\cM^1$ has surjectivity degree $\leq 2r$ and injectivity degree $\leq 2r+1$, while $\cM^0$ has surjectivity degree $\leq 2d$ and injectivity degree $\leq 2d+1$. The the result follows from  \autoref{StabDegKernelsCokernels}. 
\end{proof}

We will use the following variation of \autoref{boundingstabdeg} in the proofs of  \autoref{MCGstab} and \autoref{AutFnStab}.

\begin{proposition} \label{BoundingCoinvariants} Let $\bk$ be a finite field, and let $R$ be a field of characteristic zero. Suppose that $\cA$ is a $\Cat$--module with generation degree $\leq d$ and relation degree $\leq r$. Then the induced maps on coinvariants $$ \left(\cA_n\right)_{\G_n} \to \left(\cA_{n+1}\right)_{\G_{n+1}}$$ 
 surject for $n \geq d$ and inject for $n \geq r$. In particular these induced maps are isomorphisms for all $n \geq \max(d,r)$. 
\end{proposition}

\begin{proof} Suppose first $\cA$ is the representable $\Cat$--module $\cM(d)$. Then (as in \autoref{TheoremStabDegVICM(d)} and \autoref{VICUcoset})  a basis for the coinvariants $(\cA_n)_{\G_n}$ is given by the double cosets $ \G_n \backslash \G_n / \G_{n-d}$; these double cosets are empty for $n<d$ and a singleton set for $n \geq d$. Hence the maps $$ \left( \cM(d)_n \right)_{\G_n} \to \left(\cM(d)_{n+1}\right)_{\G_{n+1}} $$ inject for all $n \geq 0$ and surject for $n \geq d$.   

Next suppose that $\cA$ is an induced $\Cat$--module $\cM(W)$ with $W$ supported in degree $\leq d$. By \autoref{InducedVsRepresentable}  we can realize $\cA$ as both a quotient and a submodule of $\Cat$--modules of the form $ \bigoplus_{m=0}^d \cM(m)^{\oplus c_m}$.  Then by combining our results on $\cM(d)$ with the proof of \autoref{StabDegQuotientsSubmodules} in the special case $a=0$, we find that the maps on coinvariants
$$ \left( \cM(W)_n \right)_{\G_n} \to \left(\cM(W)_{n+1}\right)_{\G_{n+1}} $$
also must inject for all $n \geq 0$ and surject for $n \geq d$.   

Now consider a general $\Cat$--module $\cA$ that has a partial resolution by  induced $\Cat$--modules
$$ \cM^1 \longrightarrow \cM^0 \longrightarrow \cA$$
with $\cM^1$ generated in degree $\leq r$ and $\cM^0$ generated in degree $\leq d$.  By applying the proof of \autoref{StabDegSharp} in the special case that $a=0$, we find that the maps $$ \left(\cA_n\right)_{\G_n} \to \left(\cA_{n+1}\right)_{\G_{n+1}}$$ must inject for $n \geq r$ and surject for $n \geq d$, as claimed. 
\end{proof}

\subsection{Bounding syzygies of $\Cat$--modules over characteristic zero}

In this subsection, we will bound the degrees of the modules of higher syzygies of $\Cat$--modules presented in finite degree.

\begin{proposition}  \label{StabDegAndWeightBoundGenDeg} Let $\bk$ be a finite commutative ring, let $R$ be a field of characteristic zero, and let $\cA$ be a $\Cat$--module over $R$ of weight $\leq d$ and stability degree $\leq s$. Then $\cA$ is generated in degree $\leq (s+d)$. 
\end{proposition}

The following proof uses methods similar to those used by Church--Ellenberg--Farb \cite[Proposition 3.3.3]{CEF} to show that bounds on weight and stability degree of an FI--module imply a form of multiplicity stability. 

\begin{proof}
By \autoref{RemarkIndSurjects}, proving that $\cA$ is generated in degree at most $(s+d)$ is equivalent to showing that the induced map  
$$I_n: \mathrm{Ind}_{\G_n}^{\G_{n+1} }\cA_n \to \cA_{n+1} \qquad \text{ surjects for }n \geq s+d.$$ 
Let $C_{n+1}$ denote the cokernel of this map; our objective is to show that $C_{n+1}=0$ for $n \geq s+d$.

Recall the definition of stability degree $\leq s$: for each $a \geq 0$, $$(\cA_n)_{\G_{n-a}} \overset{\cong}\longrightarrow (\cA_{n+1})_{\G_{n+1-a}} \qquad \text{for all }n \geq s+a.$$ This map of coinvariants factors as follows, 
$$ (\cA_n)_{\G_{n-a}} \longrightarrow \left( \mathrm{Ind}_{\G_n}^{\G_{n+1} }\cA_n \right)_{\G_{n+1-a}}  \overset{(I_n)_*}\longrightarrow \left( \cA_{n+1} \right)_{\G_{n+1-a}}.$$
Since this composite map surjects for $n \geq s+a$ by assumption, it follows that the map 
$$\left( \mathrm{Ind}_{\G_n}^{\G_{n+1} }\cA_n \right)_{\G_{n+1-a}}  \overset{(I_n)_*}\longrightarrow \left( \cA_{n+1} \right)_{\G_{n+1-a}} $$ 
surjects once $n \geq s+a$, and its cokernel vanishes for any $a \geq 0$. Taking coinvariants is right exact, so this cokernel is  $(C_{n+1})_{\GL_{n+1-a}}.$

Set $a=d$. By \autoref{LemmaCoinvariantsVanish}, since $C_{n+1}$ is a quotient of  $\cA_{n+1}$ and $\cA$ has weight $\leq d$, the vanishing of $(C_{n+1})_{\G_{n+1-d}}$ for $n \geq s+d$ ensures the vanishing of $C_{n+1}$ for $n \geq s+d$. We conclude that $\cA$ is generated in degree  $ \leq (s+d)$.
\end{proof}

\begin{theorem} \label{ThmBoundedResolution} Let $\Cat$ be $\SI(\bk)$ or $\VIC(\bk)$. Let $\bk$ be a finite field, and let $R$ be a field of characteristic zero. Let $\cA$ be a $\Cat$--module over $R$ with generation degree $\leq d$ and relation degree $\leq r$. Then there exists a resolution of $\cA$ by  induced modules $\cM^k$
$$ \longrightarrow \cM^k \longrightarrow \cdots \longrightarrow \cM^2 \longrightarrow \cM^1 \longrightarrow \cM^0 \longrightarrow \cA$$ 
where $\cM^0$ is generated in degree $\leq d$, and for $k\geq1$,  $\cM^k$ is generated in degree $\leq 3^{k-1}r$. 

\end{theorem}


\begin{proof}

By assumption we have a short exact sequence $ 0 \longrightarrow \cK^0 \longrightarrow \cM^0 \longrightarrow \cA$ with $\cM^0$ an induced $\Cat$--module generated in degree $\leq d$ and the kernel $\cK^0$ generated in degree $\leq r$. So we can extend the resolution by constructing a map $\cM^1 \twoheadrightarrow \cK^0$ where $\cM^1$ is an induced $\Cat$--module generated in degree $\leq r$. 

We proceed by strong induction. Suppose  we have an exact sequence
$$ \cdots \twoheadrightarrow \cK^i \hookrightarrow \cM^i \twoheadrightarrow \cdots \twoheadrightarrow \cK^2 \hookrightarrow \cM^2 \twoheadrightarrow \cK^1 \hookrightarrow \cM^1 \twoheadrightarrow \cK^0 \hookrightarrow \cM^0 \twoheadrightarrow \cA$$ 
where $\cM^{i}$ is an induced $\Cat$--module generated in degree $\leq 3^{i-1}r$ for $i\leq k$. In particular $\cM^k$ is generated in degree $\leq 3^{k-1}r$, so it has weight $\leq 3^{k-1}r$ by \autoref{RemarkWeightGeneration} and injectivity degree $0$ by \autoref{StabDegSharp}. The kernel $$ \cK^k \hookrightarrow \cM^k \longrightarrow \cM^{k-1},$$ being a submodule of $\cM^k$, has weight $\leq 3^{k-1}r$ by \autoref{RemarkWeightSubquotients} and injectivity degree $0$ by \autoref{StabDegKernelsCokernels}. The module $\cM^k$ has stability degree $\leq (2)3^{k-1}r$ by  \autoref{StabDegSharp}, so by  \autoref{StabDegKernelsCokernels} the kernel $\cK^k$ has surjectivity degree $\leq (2)3^{k-1}r$. Then by  \autoref{StabDegAndWeightBoundGenDeg} the kernel $\cK^k$ is generated in degree $$\leq (2)3^{k-1}r + 3^{k-1}r = 3^kr.$$ This implies that we may choose $\cM^{k+1}$ to be an induced $\Cat$--module generated in degree $\leq 3^kr$, which concludes the inductive step. The resulting resolution is shown in  \autoref{TableVICResolution}.
\begin{figure}[h!]
\begin{center} { \tiny \begin{tikzpicture}
  \matrix (m) [matrix of math nodes,row sep=.6em,column sep=1.5em,minimum width=2em]
  {  \qquad \cdots  & \cM^{k+1} & \cK^k &\cM^k & \cdots & \cK^3 & \cM^3 & \cK^2 & \cM^2 & \cK^1& \cM^1 & \cK^0& \cM^0 & \cA \\
\mathrm{gen. deg.}\leq & 3^k r & 3^k r &3^{k-1}r & \cdots & 27r& 9r & 9r & 3r & 3r& r & r& d&   \\
\mathrm{weight} \leq  & 3^k r & 3^{k-1}r &3^{k-1}r & \cdots & 9r& 9r & 3r & 3r & r& r & d& d&  \\
\mathrm{inj. deg.} \leq & 0&0 &0 & \cdots & 0 & 0 & 0 &0 & 0& 0 & 0& 0 &   \\ 
\mathrm{surj. deg.} \leq & (2)3^k r & (2)3^{k-1}r &(2)3^{k-1}r & \cdots &18r & 18r & 6r & 6r & 2r& 2r & 2r& 2d &   \\};
 \path[right hook->] 
(m-1-1) edge  (m-1-2)
(m-1-3) edge  (m-1-4)
(m-1-6) edge  (m-1-7)
(m-1-8) edge  (m-1-9)
(m-1-10) edge  (m-1-11)
(m-1-12) edge  (m-1-13)
;
 \path[->>]  
(m-1-2) edge  (m-1-3)
(m-1-4) edge  (m-1-5)
(m-1-5) edge  (m-1-6)
(m-1-7) edge  (m-1-8)
(m-1-9) edge  (m-1-10)
(m-1-11) edge  (m-1-12)
(m-1-13) edge  (m-1-14)
;
\end{tikzpicture} } \end{center} 
\caption{Bounds on the syzygies of a $\VIC(\bk)$ or $\SI(\bk)$--module $\cA$ presented in finite degree. } \label{TableVICResolution} 
 \end{figure}
\end{proof}

\begin{remark} The same inductive argument given for  \autoref{ThmBoundedResolution} can also be used to show that if $\cA$ is generated in degree $\leq d$ and has injectivity degree $\leq s$, then we can construct a resolution of $\cA$ by induced $\Cat$--modules with $\cM^k$ generated in degree $\leq \max\left(3^k d, 3^{k-1}(s+d)\right)$.
\end{remark}

\begin{theorem} \label{ThmBoundedResolutionVICU} Let $\Cat$ be $\VIC^{\U}(\bk)$. Let $\bk$ be a finite field, and let $R$ be a field of characteristic zero. Let $\cA$ be a $\Cat$--module over $R$ with generation degree $\leq d$ and relation degree $\leq r$. Then there exists a resolution of $\cA$ by  induced modules $\cM^k$
$$ \longrightarrow \cM^k \longrightarrow \cdots \longrightarrow \cM^2 \longrightarrow \cM^1 \longrightarrow \cM^0 \longrightarrow \cA$$ 
where $\cM^0$ is generated in degree $\leq d$, $\cM^1$ is generated in degree $\leq r$, and for $k\geq 2$,  $\cM^k$ is generated in degree $\leq (2)3^{k-2}\max(r,d) +  3^{k-2}r+\frac12(3^{k-1}-1) $.
\end{theorem}

\begin{proof} The proof proceeds by the same argument as \autoref{ThmBoundedResolution}, using the bounds in \autoref{VICUcoset} in place of \autoref{TheoremStabDegVICM(d)}. In the case that $d \leq r$, these bounds are shown in \autoref{TableVICUResolution1}. 
 \begin{figure}[h!]
        \centering
       \begin{adjustbox}{max width=\textwidth} 
\begin{tikzpicture} 
  \matrix (m) [matrix of math nodes,row sep=.6em,column sep=1.5em,minimum width=2em]
  {  \qquad \cdots  & \cM^{k+1} & \cK^k &\cM^k & \cdots & \cK^3 & \cM^3 & \cK^2 & \cM^2 & \cK^1& \cM^1 & \cK^0& \cM^0 & \cA \\
\mathrm{gen. deg.}\leq & 3^{k}r +\frac12(3^{k}-1) & 3^{k}r +\frac12(3^{k}-1) &3^{k-1}r +\frac12(3^{k-1}-1)& \cdots & 27r+13& 9r+4 & 9r+4 & 3r+1 & 3r+1& r & r& d&   \\
\mathrm{weight} \leq  & 3^{k}r +\frac12(3^{k}-1) & 3^{k-1}r +\frac12(3^{k-1}-1) &3^{k-1}r +\frac12(3^{k-1}-1) & \cdots & 9r+4 & 9r+4 & 3r+1 & 3r+1 & r& r & d& d&  \\
\mathrm{inj. deg.} \leq & (2)3^{k}r +3^{k} &(2)3^{k-1}r +3^{k-1} &(2)3^{k-1}r +3^{k-1} & \cdots & 18r+9 & 18r+9 & 6r+3 &6r+3 & 2r+1& 2r+1 & 2d+1& 2d+1 &   \\ 
\mathrm{surj. deg.} \leq & (2)3^{k}r +(3^{k}-1) & (2)3^{k-1}r +(3^{k-1}-1) &(2)3^{k-1}r +(3^{k-1}-1)& \cdots &18r+8 & 18r+8 & 6r+2 & 6r+2 & \max(2r, 2d+1)& 2r & 2r& 2d &   \\};
 \path[right hook->] 
(m-1-1) edge  (m-1-2)
(m-1-3) edge  (m-1-4)
(m-1-6) edge  (m-1-7)
(m-1-8) edge  (m-1-9)
(m-1-10) edge  (m-1-11)
(m-1-12) edge  (m-1-13)
;
 \path[->>]  
(m-1-2) edge  (m-1-3)
(m-1-4) edge  (m-1-5)
(m-1-5) edge  (m-1-6)
(m-1-7) edge  (m-1-8)
(m-1-9) edge  (m-1-10)
(m-1-11) edge  (m-1-12)
(m-1-13) edge  (m-1-14)
;
\end{tikzpicture}
\end{adjustbox}
 \caption{Bounds on the syzygies of a $\VIC^{\U}(\bk)$--module $\cA$ with $d \leq r$. } 
  \label{TableVICUResolution1}
\end{figure}

In the case that $d>r$, the bounds are shown in \autoref{TableVICUResolution2}. 
 \begin{figure}[h!]
        \centering
       \begin{adjustbox}{max width=1\textwidth} 
       \begin{tikzpicture}
  \matrix (m) [matrix of math nodes,row sep=.6em,column sep=1.5em,minimum width=2em]
  {  \qquad \cdots  & \cM^{k+1} & \cK^k &\cM^k & \cdots  & \cM^3 & \cK^2 & \cM^2 & \cK^1& \cM^1 & \cK^0& \cM^0 & \cA \\
\mathrm{gen. deg.}\leq & (2)3^{k-1}d +  3^{k-1}r+\frac12(3^{k}-1) & (2)3^{k-1}d +  3^{k-1}r+\frac12(3^{k}-1)&(2)3^{k-2}d +  3^{k-2}r+\frac12(3^{k-1}-1)& \cdots & 6d+3r+4 & 6d+3r+4 & 2d+r+1 & 2d+r+1& r & r& d&   \\
\mathrm{weight} \leq  & (2)3^{k-1}d +  3^{k-1}r+\frac12(3^{k}-1) & (2)3^{k-2}d +  3^{k-2}r+\frac12(3^{k-1}-1) &(2)3^{k-2}d +  3^{k-2}r+\frac12(3^{k-1}-1) & \cdots  & 6d+3r+4& 2d+r+1 & 2d+r+1 & r& r & d& d&  \\
\mathrm{inj. deg.} \leq & (4)3^{k-1}d +  (2)3^{k-1}r+3^{k}&(4)3^{k-2}d +  (2)3^{k-2}r+3^{k-1} &(4)3^{k-2}d +  (2)3^{k-2}r+3^{k-1}& \cdots   & 12d+6r+9 & 4d+2r+3 &4d+2r+3 & 2r+1& 2r+1 & 2d+1& 2d+1 &   \\ 
\mathrm{surj. deg.} \leq & (4)3^{k-1}d +  (2)3^{k-1}r+(3^{k}-1) & (4)3^{k-2}d +  (2)3^{k-2}r+(3^{k-1}-1)&(4)3^{k-2}d +  (2)3^{k-2}r+(3^{k-1}-1)& \cdots  & 12d+6r+8 & 4d+2r+2 & 4d+2r+2 & 2d+1& 2r & 2r& 2d &   \\};
 \path[right hook->] 
(m-1-1) edge  (m-1-2)
(m-1-3) edge  (m-1-4)
(m-1-5) edge  (m-1-6)
(m-1-7) edge  (m-1-8)
(m-1-9) edge  (m-1-10)
(m-1-11) edge  (m-1-12)
;
 \path[->>]  
(m-1-2) edge  (m-1-3)
(m-1-4) edge  (m-1-5)
(m-1-6) edge  (m-1-7)
(m-1-8) edge  (m-1-9)
(m-1-10) edge  (m-1-11)
(m-1-12) edge  (m-1-13)
;
\end{tikzpicture} 
\end{adjustbox}
 \caption{Bounds on the syzygies of a $\VIC^{\U}(\bk)$--module $\cA$ with $d > r$. } 
 \label{TableVICUResolution2}
\end{figure}
\end{proof}

\subsection{$\Cat$--module homology} \label{SectionCModuleHomology}

This subsection is not needed to prove our results about congruence subgroups of mapping class groups and automorphism groups of free groups. We include it because it allows us to reformulate \autoref{ThmBoundedResolution} and \autoref{ThmBoundedResolutionVICU} in a way that does not explicitly reference resolutions. In analogy to the theory of $\FI$--homology developed by Church, Ellenberg, and Farb \cite{CEF, CE}, we make the following definition. 

\begin{definition} \label{DefnH0} Define a functor $H_0^{\Cat}: \Cat$--Mod $\to \Cat$--Mod as the quotient $$ H_0^{\Cat}(\cA)_{V} = \frac{\cA_V}{\langle \; f_*(W)\; | \; f \in \Hom_{\Cat}(W, V), \; \dim_{\bk}W < \dim_{\bk}V \rangle }$$ 
Equivalently, $H_0^{\Cat}(\cA)$ is the largest $\Cat$--module quotient of $\cA$ such that all non-isomorphism $\Cat$ morphisms act by zero. 
By abuse of notation, we also write $H_0^{\Cat}$ to denote the composition of $H^{\Cat}_0$ with the forgetful functor $\Cat$--Mod $\to$ $\CB$--Mod.
\end{definition}
\begin{remark} \label{RemarkComputingH0}
We remark that, since every $\VIC^{\U}(\bk)$ morphism $f: W\to V$ with $\dim_{\bk}W < \dim_{\bk}V$ factors through a morphism $Z \to V$ with $ \dim_{\bk}Z = \dim_{\bk}V -1$, it suffices to take 
$$ H_0^{\VIC^{\U}}(\cA)_{V} = \frac{\cA_V}{\langle \; f_*(Z) \; | \; f \in \Hom_{\VIC^{\U}}(Z, V), \; \dim_{\bk}Z= \dim_{\bk}V-1 \rangle }.$$ 
Similarly, 
$$ H_0^{\SI}(\cA)_{V} = \frac{\cA_V}{\langle \; f_*(Z) \; | \; f \in \Hom_{\SI}(Z, V), \; \dim_{\bk}Z= \dim_{\bk}V-2 \rangle }.$$ 
\end{remark}

The following proposition summarizes some properties of the functor $H_0^{\Cat}$. Several parts are analogous to \cite[Definition 2.3.7 and Remark 2.3.8]{CEF}. 

\begin{proposition} \label{H0Properties} Let $R$ be a commutative ring and consider the categories of $\Cat$--modules and $\CB$--modules over $R$. The functor $H_0^{\Cat}: \Cat$--Mod $\to$ $\CB$--Mod satisfies the following. 
\begin{enumerate}
\item  \label{RemarkH0Generation} A $\Cat$--module $\cA$ is generated in degree $\leq d$ if and only if $H_0^{\Cat}(\cA)$ is supported in degree $\leq d$. 
\item \label{splittings} Suppose $\bk$ is a finite commutative ring, $R$ is a field of characteristic zero, and $\cA$ is a $\Cat$--module. There are (non-canonical) splittings 
$$H_0^{\Cat}(\cA)_n \to \cA_n \qquad \text{in each degree $n$.}$$
\item \label{surjections} If $\bk$ is a finite commutative ring and $R$ is a field of characteristic zero, then any $\Cat$--module $\cA$ can be realized as a quotient of the induced module $$ \cM(H_0^{\Cat}(\cA)) \twoheadrightarrow \cA .$$
For general commutative rings $R$ and $\bk$, the $\Cat$--module $\cA$ can be realized as a quotient of the induced module $$ \cM\left( \{ \cA_n \; | \; n \in \text{support}(H_0^{\Cat}(\cA)) \} \right) \twoheadrightarrow \cA .$$
\item \label{inverse} The functor $H_0^{\Cat}$ is a left inverse to the functor $\cM$, that is, 
$$ H_0^{\Cat}( \cM(W)) = W \qquad \text{ for all $\CB$--modules $W$.} $$ 
\item  \label{adjoint} The functor $H_0^{\Cat}$ is the left adjoint to  the inclusion of categories $$\iota: \CB\text{--Mod} \to \Cat\text{--Mod},$$ where $\iota$ is defined such that non-isomorphism $\Cat$ morphisms act on $\iota(W)$ by zero. 
\item \label{H0Exact} The functor $H^{\Cat}_0$ is right exact. 
Hence, the same is true of  $H^{\Cat}_0$ when viewed as a functor $$H^{\Cat}_0: \Cat\text{--Mod} \to \Cat\text{--Mod}.$$
\end{enumerate} 
\end{proposition}

\begin{proof}
By definition, $H_0^{\Cat}(\cA)_n=0$ only if the $R[\G_n]$--module $\cA_n$ is generated by the image of $\cA_{n-1}$. Hence \autoref{RemarkH0Generation} follows from \autoref{RemarkIndSurjects} \autoref{b}. 
 \autoref{splittings} follows because $R[\G_n]$ is semi-simple by assumption, so the natural surjections $\cA_n \to H_0^{\Cat}(\cA)_n$ split. The map $\{ H_0^{\Cat}(\cA)_n\} \to \{\cA_n\}$ of $\CB$--modules constructed in \autoref{splittings} then induces the map of $\Cat$--modules  $\cM(H_0^{\Cat}(\cA)) \to \cA$ of \autoref{surjections}, and (as in the equivalence of \autoref{RemarkIndSurjects} \autoref{d} and \autoref{e}) it is not difficult to deduce from the definition of $H_0^{\Cat}$ that this map must surject. More generally, there is a surjective map of $\C$--modules 
 $$ \cM\left( \{ \cA_n \; | \; n \in \text{support}(H_0^{\Cat}(\cA)) \} \right) \longrightarrow \cA$$
 by an argument similar to the proof of \autoref{RemarkIndSurjects}  \autoref{e}. 
 
 \autoref{inverse} can be verified directly from the formula for $\cM(W)$. \autoref{adjoint} follows as in \cite[Definition 2.3.7 and Remark 2.3.8]{CEF}. To deduce  \autoref{H0Exact}, observe that  $H_0^{\Cat}: \Cat$--Mod $\to$ $\CB$--Mod is the left adjoint to $\iota$, and therefore right exact \cite[Theorem 2.6.1]{Weibel}. Since exactness is defined pointwise on $\Cat$--modules, the same result implies that  $H_0^{\Cat}$ is exact as a functor $\Cat$--Mod $\to$ $\Cat$--Mod. 
\end{proof}

By  \autoref{H0Properties} \ref{H0Exact}, we may make the following definition.

\begin{definition} \label{DefnHp} Define the functors $H_k^{\C}: \Cat$--Mod $\to  \Cat$--Mod   to be the left derived functors of $H_0^{\Cat}$. 
\end{definition}

To compute the $\Cat$--homology of a $\Cat$--module $\cA$, we may take an acyclic resolution $\cP^* \to \cA$, apply $H^{\C}_0$ to each term and pass to homology. The following proposition shows that we take the terms $\cP^i$ to be any induced modules. 

\begin{proposition} \label{SharpAcyclic} Induced $\Cat$--modules over $R$ are $H_*^{\Cat}$--acyclic. 
\end{proposition}

\begin{proof}
Let $W$ be a $\CB$--module. It suffices to show that $H_k^{\Cat} (\cM(W)) = 0$ for all $k>0$. Let $$ \cdots \longrightarrow P^2 \longrightarrow P^1 \longrightarrow P^0 \longrightarrow W$$ be a projective resolution of $W$ by $\CB$--modules. Since $\cM$ is exact by  \autoref{MExact}, we can promote this resolution to a resolution of $\cM(W)$ by induced $\Cat$--modules 
$$ \cdots \longrightarrow \cM(P^2) \longrightarrow \cM(P^1) \longrightarrow \cM(P^0) \longrightarrow \cM(W).$$
By  \autoref{CorMWProjective}, this is a projective resolution. Applying $H^{\Cat}_0$, however, recovers our original resolution
$$ \cdots \longrightarrow P^2 \longrightarrow P^1 \longrightarrow P^0.$$ This resolution is exact by construction, and so we find $H_k^{\Cat} (\cM(W)) = H_k(P^*) =0$ for $k>0$. 
\end{proof}

\begin{proposition} Let $R$ be a field of characteristic zero and $\bk$ a finite field. Let  $\cA$ be a $\Cat$--module over $R$ generated in degree $\leq d$ and related in degree $\leq r$. Then $H_0^{\Cat} (\cA)_V$ vanishes for  $\dim_\bk V > d$ and $H_1^{\Cat} (\cA)_V$ vanishes for  $\dim_\bk V > r$.  
\begin{itemize}
\item If $\Cat$ is $\SI(\bk)$ or $\VIC(\bk)$,  then for $k \geq 1$, the groups  $H_k^{\Cat} (\cA)_V$ vanish once $\dim_\bk V > 3^{k-1}r$. 
\item If $\Cat$ is $\VIC^{\U}(\bk)$, then for $k \geq 2$, the groups  $H_k^{\Cat} (\cA)_V$ vanish once $$\dim_\bk V >  (2)3^{k-2}\max(r,d) +  3^{k-2}r+\frac12(3^{k-1}-1) .$$
\end{itemize}
 \label{regC}
\end{proposition}


\begin{proof} By \autoref{SharpAcyclic}, we can compute $H_k^{\Cat} (\cA)_d$ by resolving $\cA$ by induced $\Cat$--modules, applying the functor $H^{\Cat}_0$ and taking homology. The result follows from applying $H^{\Cat}_0$  to the resolution described in  \autoref{ThmBoundedResolution} or  \autoref{ThmBoundedResolutionVICU}. 
\end{proof}

The following proposition relates the vanishing of  $H_0^{\Cat} (\cA)_n$ and $H_1^{\Cat} (\cA)_n$ to the generation and relation degree of a $\Cat$--module $\cA$. 

\begin{proposition} \label{presentLemma} Suppose that $\cA$ is a $\Cat$--module such that $H_0^{\Cat} (\cA)_n=0$ for $n > d$ and $H_1^{\Cat} (\cA)_n=0$ for $n > r$. Then $\cA$ is generated in degree $\leq d$ and related in degree $\leq \max(r,d)$. 
\end{proposition}

\begin{proof}  

 \autoref{H0Properties} implies that $\cA$ is generated in degree $\leq d$ and that we can find a short exact sequence $$0 \to \cK \to \cM \to \cA \to 0$$ with $\cM$ an induced $\Cat$--module which is generated in degree $\leq d$.  Consider the associated long exact sequence on homology $$ \cdots \longrightarrow H_1^{\Cat} (\cA)_n \longrightarrow  H_0^{\Cat} (\cK)_n \longrightarrow H_0^{\Cat} (\cM)_n \longrightarrow H_0^{\Cat} (\cA)_n \longrightarrow 0.$$ Since $H_1^{\Cat} (\cA)_n = 0$ for $n > r$ and $H_0^{\Cat} (\cM)_n=0$ for $n>d$, it follows that $H_0^{\Cat} (\cK)_n$ must vanish for $n> \max(r,d)$. The claim follows by  \autoref{H0Properties} \ref{RemarkH0Generation}.
\end{proof}





Combining \autoref{regC} and \autoref{presentLemma} establishes \autoref{regSI} and \autoref{regVIC}. 

The following corollaries were suggested to us by Eric Ramos. We state these without explicit ranges although the proofs we give can easily be made effective. 

\begin{corollary} \label{abCat} Let $\bk$ be a finite field and $R$ a field of characteristic zero. Let $\Cat$ be one of the categories $\SI(\bk)$ or $\VIC^{\U}(\bk)$. Then the category of $\Cat$-modules presented in finite degree is an abelian category. 
\end{corollary}

\begin{proof}
Let $f:\cA \m \cB$ be a map between $\Cat$--modules presented in finite degree. We must check that $\ker(f)$ and $\coker(f)$ are presented in finite degree. Note that without any assumptions on $R$ and $\bk$, it is true that the cokernel of a map of $\Cat$-modules presented in finite degree is presented in finite degree. 

By \autoref{regSI} in the case of $\SI$ and \autoref{regVIC} in the case of $\VIC^{\U}$, we see that $H_2^\Cat(\coker(f))_n \cong 0$ for $n$ sufficiently large. By considering the long exact sequence of $\Cat$--homology groups associated to the short exact sequence $$ 0 \m \im(f) \m B \m \coker(f) \m 0,$$ we see that $H_1^\Cat(\im(f))_n \cong H_0^\Cat(\im(f))_n \cong 0$ for $n$ sufficiently large. \autoref{regSI} and \autoref{regVIC} imply that $H_2^\Cat(\im(f))_n \cong 0$ for $n$ sufficiently large. By considering the long exact sequence of $\Cat$--homology groups associated to the short exact sequence $$ 0 \m \ker(f) \m A \m \im(f) \m 0,$$ we see that $\ker(f)$ is presented in finite degree. 

\end{proof}

\begin{corollary} Let $\bk$ be a finite field and $R$ a field of characteristic zero. Let $\Cat$ be one of the categories $\SI(\bk)$ or $\VIC^{\U}(\bk)$. Let $ \cB$ be a $\Cat$--module and $\cA$ a $\Cat$--submodule. If $\cA$ has finite generation degree and $\cB$ has finite presentation degree, then $\cA$ has finite presentation degree. 
\end{corollary}

\begin{proof}
Let $\cK$ denote $\cB/\cA$. By considering the long exact sequence in $\Cat$-homology associated to $$0 \m \cA \m \cB \m \cK \m 0,$$ we see that $\cK$ has finite presentation degree. Thus, by \autoref{regSI} and \autoref{regVIC}, $H_2^\Cat(\cK)_n \cong 0$ for $n$ sufficiently large. By again considering the long exact sequence in $\Cat$-homology associated to $$0 \m \cA \m \cB \m \cK \m 0,$$ we see that $\cA$ has finite presentation degree. 
\end{proof}

\section{Representation stability results}
\label{Representationstabilityresults}

In this section, we apply the algebraic tools developed in the previous section to prove our representation stability theorems. 


\subsection{Central stability homology}

Central stability homology is an invariant of modules over categories such as $\SI(\bk)$ or $\VIC(\bk)$. In the context of $\SI(\bk)$--modules and $\VIC(\bk)$--modules, it was introduced by Putman--Sam \cite{PS}, though the name \emph{central stability homology} is due to Patzt \cite{Pa2}, based on earlier terminology in the work of Putman \cite{Pu}. In this subsection, we describe basic properties of central stability homology. After a draft of this paper was circulated, we were informed that many of the results of this subsection were independently established by Patzt \cite{Pa2}. In the interest of space, we will not reprove these properties.

Let $\Delta'$ denote the \emph{augmented semi-simplicial category}, the category of finite ordered sets and order-preserving injections. We will realize $\Delta'$ as a subcategory of  $\SI(\bk)$ and of $\VIC(\bk)$ by inclusions $s: \Delta' \m \SI(\bk)$ and $v: \Delta' \m \VIC(\bk)$ defined as follows. Given an ordered set $X$, let $s(X)$ be the free $\bk$--module on $X \sqcup X$ with $X \sqcup X$ a symplectic basis. Injective maps of sets induce symplectic embeddings by extending linearly. Let $v(X)$ be the free $\bk$--module on $X$. Given an order-preserving injection $\iota:X \m Y$, let $T:v(X) \m v(Y)$ be the linear map induced by $\iota$ and let $C$ be span $(Y-\im(f))$ in $v(Y)$. Define $v(\iota)$ to be $(T,C)$.

\begin{definition}
Let $\cA$ be an $\SI(\bk)$--module. We now define a augmented semi-simplicial $\SI(\bk)$--module $\CC_\bullet(\cA)$ whose value on an ordered set $X$ and a symplectic $\bk$--module $V$ is given by the formula 
\[\CC_X(\cA)_V= \bigoplus_{T \in \Hom_{\SI(\bk)}(s(X),V)} \cA(\im(T)^\perp). \] 
Composition induces the augmented semi-simplicial and $\SI(\bk)$--module structure.


Similarly for $\cA$ a $\VIC^{\U}(\bk)$--module, we define $\CC_\bullet(\cA)$ by the formula \[ \CC_X(\cA)_V= \bigoplus_{(T,C) \in \Hom_{\VIC(\bk)}(v(X),V)} \cA(C). \]

Let $\CC_i(\cA)_V$ denote $\CC_X(\cA)_V$ for $X=\{0,\ldots,i\}$. Let $\CC_*(\cA)_V$ denote the chain complex formed by taking the alternating sum of the face maps and let  $\HH_i(\cA)_V$ denote its homology $H_i(\CC_*(\cA)_V)$. We call the chain complex $\CC_*(\cA)$ the \emph{central stability chains} on $\cA$ and call its homology $\HH_*(\cA)$ the \emph{central stability homology}. 
\end{definition}

Central stability homology is closely related to $\SI$--homology and $\VIC^{\U}$--homology, and both control the generation degrees of the modules of syzygies. 

Patzt \cite[Theorem 5.7]{Pa2} gave a general criterion for results of the form of the following \autoref{ThmPatztBoundedResolution} to hold for a broad class of categories $\Cat$. He verifies the criterion for the categories $\SI(\bk)$ and $\VIC(\bk)$ \cite[Remark 5.6]{Pa2}.  Miller--Patzt--Wilson \cite[Proposition 3.14]{MPW} verified the criterion in the case $\Cat=\VIC^{\U}(\bk)$.

\begin{theorem}[Patzt {\cite[Theorem 5.7]{Pa2}, Miller--Patzt--Wilson \cite[Proposition 3.14]{MPW}}] \label{ThmPatztBoundedResolution} \qquad 

\noindent Let $\cA$ be an $\SI(\bk)$--module with $\bk$ a field. Let $d_0, \ldots, d_k$ be integers with $d_{i+1} -d_i \geq 3$. Then the following are equivalent. \begin{enumerate}
\item There is an exact sequence of $\SI(\bk)$--modules $$\cW^k \m \cW^{k-1} \m \ldots \m \cW^0 \m \cA \m 0$$ with $\cW^i$ induced and generated in degrees $\leq d_i$. 
\item $\HH_i(\cA)_n=0$ for $n>d_{i+1}$ for all $i <k$. 
\end{enumerate}

\noindent Let $\cA$ be a $\VIC^{\U}(\bk)$--module with $\bk$ a field. Let $d_0, \ldots, d_k$ be integers with $d_{i+1} -d_i \geq 2$. Then the following are equivalent. \begin{enumerate}
\item There is an exact sequence of $\VIC^{\U}(\bk)$--modules $$\cW^k \m \cW^{k-1} \m \ldots \m \cW^0 \m \cA \m 0$$ with $\cW^i$ induced and generated in degrees $\leq d_i$. 
\item $\HH_i(\cA)_n=0$ for $n>d_{i+1}$ for all $i <k$. 
\end{enumerate}

\end{theorem}

The following is a reformulation of work of Randal-Williams--Wahl \cite[Lemma 5.9]{RWW} and Mirzaii--van der Kallen \cite[Theorem 7.4]{MvdK}. See also Miller--Patzt--Wilson \cite[Proposition 3.14]{MPW}. It is a slight sharpening of the above theorem for the induced module $\cM(0)$.

\begin{proposition}[Patzt {\cite[Remark 5.6]{Pa2}}] \label{PropM(0)}
Let $\bk$ be a field. Then $\HH_i(\cM^{\SI}(0))_n \cong 0$ for $n> 2i+3 $ and $\HH_i(\cM^{\VIC^{\U}}(0))_n \cong 0$ for $n> 2i+2$.
\end{proposition}

\subsection{Stability for congruence subgroups}

\subsubsection{Congruence subgroups of mapping class groups}

Putman--Sam \cite[Corollary 6.22]{PS} observe that the representations $H_i(\Mod(\Sigma_{g,1},p);R)$ assemble to form an $\SI(\Z/p\Z)$--module over $R$. We denote this $\SI(\Z/p\Z)$--module by $H_i(\Mod(\Sigma,p);R)$.

We prove our results on congruence subgroups using spectral sequences introduced by Putman--Sam \cite{PS}. The following is implicit in the proof of \cite[Theorem K]{PS} and builds on \cite[Theorem 5.13, Lemma 6.24, Theorem 6.25]{PS}. See also Patzt \cite[Corollary 8.5]{Pa2} and Miller--Patzt--Wilson \cite[Proposition 3.38]{MPW}.

\begin{theorem}[Putman--Sam {\cite{PS}}] \label{SIss}
For each $g>0$, there is a homologically graded spectral sequence $E^r_{a,b}(g)$ satisfying the following properties. 
\begin{enumerate}
 \item \label{ItemEr} $E^r_{a,b}(g) \cong 0$ for $a<-1$ or $b <0$.
  \item \label{ItemE2}  $E^2_{a,b}(g) \cong \HH_a(H_b(\Mod(\Sigma,p);R))_g$.
   \item \label{ItemEinfty} $E^\infty_{a,b}(g) \cong 0$ for $a+b \leq \frac{g-3}{2}$. 
\end{enumerate}
The $E^2$ page is illustrated in  \autoref{E2pageSI}. 
\end{theorem}
\begin{figure}[h!]    \centering \begin{tikzpicture} \footnotesize
  \matrix (m) [matrix of math nodes,
    nodes in empty cells,nodes={minimum width=3ex,
    minimum height=5ex,outer sep=2pt},
 column sep=3ex,row sep=3ex]{ 
 3    &  \HH_{-1}(H_3(\Mod(\Sigma,p);R))_g &  \HH_{0}(H_3(\Mod(\Sigma,p);R))_g  & \HH_{1}(H_3(\Mod(\Sigma,p);R))_g  &  \HH_{2}(H_3(\Mod(\Sigma,p);R))_g & \\       
 2    &  \HH_{-1}(H_2(\Mod(\Sigma,p);R))_g & \HH_{0}(H_2(\Mod(\Sigma,p);R))_g&  \HH_{1}(H_2(\Mod(\Sigma,p);R))_g  &\  \HH_{2}(H_2(\Mod(\Sigma,p);R))_g & \\           
1     &   \HH_{-1}(H_1(\Mod(\Sigma,p);R))_g &  \HH_{0}(H_1(\Mod(\Sigma,p);R))_g & \HH_{1}(H_1(\Mod(\Sigma,p);R))_g  &   \HH_{2}(H_1(\Mod(\Sigma,p);R))_g & \\                  
 0     &  \HH_{-1}(H_0(\Mod(\Sigma,p);R))_g & \HH_{0}(H_0(\Mod(\Sigma,p);R))_g &  \HH_{1}(H_0(\Mod(\Sigma,p);R))_g  &  \HH_{2}(H_0(\Mod(\Sigma,p);R))_g & \\       
 \quad\strut &   -1  &  0  &  1  & 2  &\\}; 
 \draw[thick] (m-1-1.east) -- (m-5-1.east) ;
 \draw[thick] (m-5-1.north) -- (m-5-6.north east) ;
\end{tikzpicture}
\caption{$E^2_{a,b}(g)$. } \label{E2pageSI}
\end{figure}

We now prove the following strengthening of \autoref{TheoremMod}. 

\begin{theorem} \label{detailedModp}
Let $R$ be a field of characteristic zero and $p$ be a prime. The $\SI(\Z/p\Z)$--module $$H_0(\Mod(\Sigma,p);R) \cong \cM^{\SI}(0)$$ is generated in degree $\leq 0$ and has no relations. The $\SI(\Z/p\Z)$--module $H_1(\Mod(\Sigma,p);R)$ is generated in degree $\leq 5$ and related in degree $\leq 8$. For $i >1$, the $\SI(\Z/p\Z)$--module $H_i(\Mod(\Sigma,p);R)$ is generated in degree $\leq (8)3^{2i-3}$ and related in degree $\leq (8)3^{2i-2}$.
\end{theorem}


\begin{proof}  We proceed by induction on $i$. Since $H_0(\Mod(\Sigma,p);R) \cong \cM^{\SI}(0),$ by \autoref{PropM(0)}, 
 $$E^2_{a,0} \cong \HH_a(H_0(\Mod(\Sigma,p);R) )_g = 0 \qquad \text{for $g>2a+3$.}$$ Now consider the $b=1$ row of the spectral sequence, which corresponds to the homology of $\Mod(\Sigma,p)$ in degree $i=1$. This row requires some additional care, so we will show explicitly how to bound the vanishing of these central stability homology groups. Once $g \geq 3$, $E^\infty_{-1,1}(g) = 0$ by \autoref{SIss}. But for $g >5$, the group $E^2_{1,0}(g) \cong  \HH_{1}(H_0(\Mod(\Sigma,p);R))_g$ vanishes, and so in this range there are no nonzero differentials into or out of the groups $E^r_{-1,1}(g)$ for any $r \geq 2$. Thus $$E^2_{-1,1}(g) \cong \HH_{-1}(H_1(\Mod(\Sigma,p);R))_g = 0 \qquad \text{for $g >5$}.$$
 Similarly $E^\infty_{0,1}(g) = 0$ for $g \geq 5$ and for $r \geq 2$ the domain $E^r_{2,0}(g)$ of the only possible nonzero differential to or from $E^r_{0,1}(g)$ vanishes for $g >7$. 
 Thus $$\HH_{0}(H_1(\Mod(\Sigma,p);R))_g = 0 \qquad \text{for $g >7$}.$$
 See \autoref{E2PageVanishing}. 
 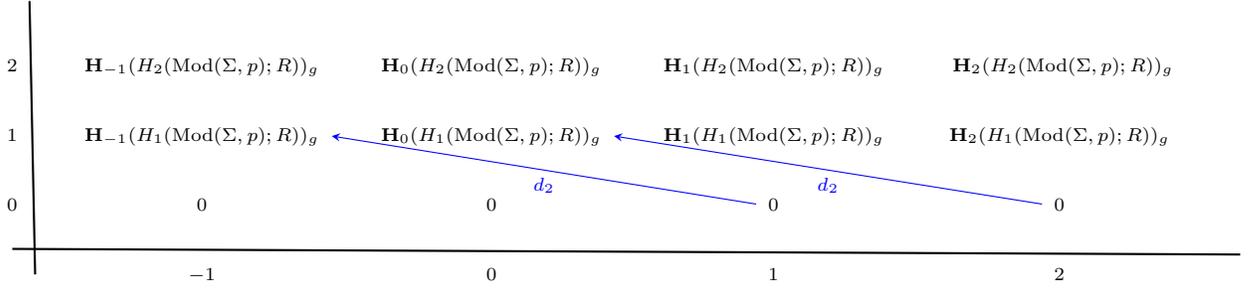
\begin{figure}[h!]    \centering \begin{tikzpicture} \scriptsize 
  \matrix (m) [matrix of math nodes,
    nodes in empty cells,nodes={minimum width=3ex,
    minimum height=5ex,outer sep=2pt},
    column sep=6ex,row sep=3ex]{ 
                &      &     &     &\strut& & & &  \\    
 2    &  \HH_{-1}(H_2(\Mod(\Sigma,p);R))_g & \HH_{0}(H_2(\Mod(\Sigma,p);R))_g&  \HH_{1}(H_2(\Mod(\Sigma,p);R))_g  &\  \HH_{2}(H_2(\Mod(\Sigma,p);R))_g & \\          
1     &   \HH_{-1}(H_1(\Mod(\Sigma,p);R))_g &  \HH_{0}(H_1(\Mod(\Sigma,p);R))_g & \HH_{1}(H_1(\Mod(\Sigma,p);R))_g  &   \HH_{2}(H_1(\Mod(\Sigma,p);R))_g & \\             
 0     &  0  & 0   & 0   &0 &    \\       
 \quad\strut &   -1  &  0  &  1  & 2  &   \\}; 
 
 \draw[-stealth, blue] (m-4-4.west) -- (m-3-2.east) node [midway,below] {$d_2$};
 
 \draw[-stealth, blue] (m-4-5.west) -- (m-3-3.east) node [midway,below] {$d_2$};

\draw[thick] (m-1-1.east) -- (m-5-1.east) ;
\draw[thick] (m-5-1.north) -- (m-5-6.north) ;

\end{tikzpicture}
\caption{Page $E^2_{p,q}(g)$ for $g \ge 8$.} \label{E2PageVanishing}
\end{figure}  

 If we replace the condition $g >7$ with the weaker condition $g > 8$, then these two central stabiltiy homology groups satisfy the hypotheses of Patzt's \autoref{ThmPatztBoundedResolution}, and we obtain a partial resolution of induced $\SI(\Z/p\Z)$--modules $$ \cM^1 \longrightarrow \cM^0  \longrightarrow H_1(\Mod(\Sigma,p);R) \longrightarrow 0$$
 with $\cM^0$ generated in degree $\leq 5$ and $\cM^1$ generated in degree $\leq 8$. 

 We now proceed with the inductive step. Suppose that $j > 1$ and that $H_i(\Mod(\Sigma,p);R)$ is generated in degree $\max(5, (8)3^{2i-3})$ and related in degree $\leq (8)3^{2i-2}$ for all $1\leq i <j$.  Then \autoref{ThmPatztBoundedResolution} implies that there is a partial resolution of $H_i(\Mod(\Sigma,p);R)$ by induced $\SI(\Z/p\Z)$--modules with term $\cM^0$ generated in degree $\max(5, (8)3^{2i-3})$ and $\cM^1$ generated in degree $\leq (8)3^{2i-2}$. It follows by \autoref{ThmBoundedResolution} that we can extend this partial resolution to a resolution by induced modules with term $\cM^k$ generated in degree $\leq (8)(3^{2i-2})(3^{k-1})$. Then \autoref{ThmPatztBoundedResolution} implies that $\HH_{k}(H_i(\Mod(\Sigma,p);R))_g$ vanishes for $g > (8)(3^{2i-2})(3^{k})$. Small values of these bounds are shown in  \autoref{E2PageVanishingRanges}, with some differentials superimposed.  
 \begin{figure}[h!]    \centering \begin{tikzpicture} \scriptsize 
  \matrix (m) [matrix of math nodes,
    nodes in empty cells,nodes={minimum width=3ex,
    minimum height=5ex,outer sep=2pt},
    column sep=6ex,row sep=3ex]{ 
4 &  8(3^5) & 8(3^6) &  8(3^7) &  8(3^8) & 8(3^9) & 8(3^{10}) \\ 
3 &  8(3^3) & 8(3^4) &  8(3^5) &  8(3^6) & 8(3^7) & 8(3^8) \\ 
 2    &  8(3) &  8(3^2)  & 8(3^3) & 8(3^4) &  8(3^5) &  8(3^6) & \\          
1     &   5 & 8  & 8(3)  & 8(3^2)  & 8(3^3) & 8(3^4) &   \\             
 0     &  0  & -1   & -1   &-1 & -1 & -1 &   \\       
 \quad\strut &   -1  &  0  &  1  & 2  & 3 & 4 &  \\}; 

\draw[thick] (m-1-1.east) -- (m-6-1.east) ;
\draw[thick] (m-6-1.north) -- (m-6-8.north) ;

 \draw[-stealth, blue] (m-2-4.west) -- (m-1-2.east); 
 \draw[-stealth, blue] (m-3-5.west) -- (m-1-2.east); 
  \draw[-stealth, blue] (m-4-6.west) -- (m-1-2.east); 
    \draw[-stealth, blue] (m-5-7.west) -- (m-1-2.east); 
    
\end{tikzpicture}
\caption{$E^2_{a,b}(g)$ vanishes at each point once $g$ is strictly greater than the stated value.} \label{E2PageVanishingRanges}
\end{figure}
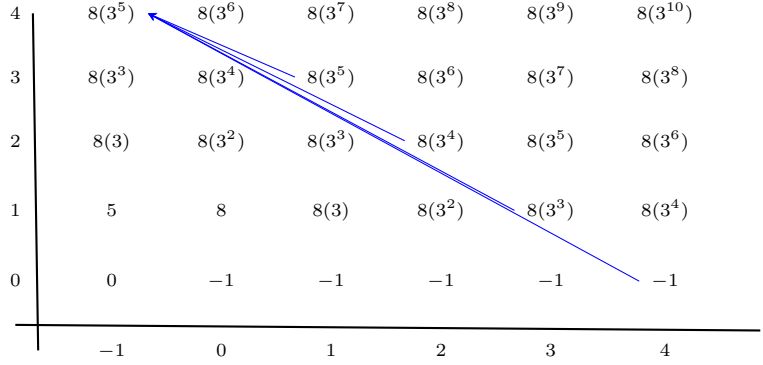  
 
  In particular, $E^2_{-1+r, j-r+1}(g) = \HH_{-1+r}(H_ {j-r+1}(\Mod(\Sigma,p);R))_g$ vanishes for $g > (8)(3^{2(j-r+1)-2})(3^{r-1})$ for $2 \leq r \leq j+1$, so there are no nonzero differentials to or from $E^r_{-1, j}(g)$ once $r\geq 2 $ and $g > (8)(3^{2j-3})$. Since $E^{\infty}_{-1, j}(g)=0$ in this range, we conclude that 
  $$E^2_{-1, j}(g) = \HH_{-1}(H_{j}(\Mod(\Sigma,p);R))_g = 0 \qquad \text{for $g > (8)(3^{2j-3})$}.$$ 
  Similarly $E^2_{r,  j-r+1}(g) = \HH_{r}(H_{j-r+1}(\Mod(\Sigma,p);R))_g$ vanishes for $g > (8)(3^{2(j-r+1)-2})(3^{r})$  for $2 \leq r \leq j+1$, so there are no nonzero differentials to or from $E^r_{0, j}(g)$ once $r \geq 2$ and $g > (8)(3^{2j-2})$. Again $E^{\infty}_{-1, j}(g)=0$ in this range, so we conclude that 
  $$E^2_{0, j}(g) = \HH_{0}(H_{j}(\Mod(\Sigma,p);R))_g = 0 \qquad \text{for $g > (8)(3^{2j-2})$}.$$ 
 Finally, \autoref{ThmPatztBoundedResolution} then implies that $H_{j}(\Mod(\Sigma,p);R)$ is generated in degree $\leq (8)(3^{2j-3})$ and related in degree $\leq (8)(3^{2j-2})$, which concludes the inductive step.
\end{proof}

\subsubsection{Congruence subgroups of automorphism groups of free products}

Putman--Sam \cite[Corollary 6.7]{PS} observed that the representations $H_i(\Aut(F_n,p);R)$ assemble to form a $\VIC^{\pm}(\Z/p\Z)$--module over $R$ which we will denote by $H_i(\Aut(F,p);R)$. Implicitly in the proof of \cite[Theorem I]{PS} and building on \cite[Theorem 5.13, Lemma 6.8, Theorem 6.9]{PS}, Putman--Sam proved the following. 

\begin{theorem}[Putman--Sam {\cite{PS}}] \label{SSaut}
For all $n$, there is a homologically graded spectral sequence $E^r_{a,b}(n)$ satisfying the following properties. 
\begin{enumerate}
 \item $E^r_{a,b}(n) \cong 0$ for $a<-1$ or $b <0$.
  \item $E^2_{a,b}(n) \cong \HH_a(H_b(\Aut(F,p);R))_n$.
   \item $E^\infty_{a,b}(n) \cong 0$ for $a+b \leq \frac{n-3}{2}$. \label{connectivityautomorphism}
 \end{enumerate}
 
\end{theorem}
%
%
%

The following implies \autoref{thmB}.

\begin{theorem} \label{detailedAutp}
Let $R$ be a field of characteristic zero and $p$ be a prime. Then the $\VIC^{\pm}(\Z/p\Z)$--module $$H_0(\Aut(F,p);R) \cong \cM^{\VIC}(0)$$ is generated in degree $\leq 0$ and has no relations. 
The $\VIC^{\pm}(\Z/p\Z)$--module $H_1(\Aut(F,p);R)$ is generated in degree $\leq 4$ and related in degree $\leq 6$. For $i >1$, the $\VIC^{\pm}(\Z/p\Z)$--module $H_i(\Aut(F,p);R)$ is generated in degree $\leq (\frac{13}{2})3^{2i-3}-\frac12$ and related in degree $\leq (\frac{13}{2})3^{2i-2}-\frac12$.
\end{theorem}



\begin{proof} Since $H_0(\Aut(F,p);R) \cong \cM(0)$, these groups are generated in degree $\leq 0$ and have no relations. 
The bottom row of the $E^2(n)$ page, $E^2_{k,0}(n) \cong \HH_k(H_0(\Aut(F,p);R))$, vanishes for $n > 2k+2$ by \autoref{PropM(0)}. The groups $E^r_{-1,1}(n)$ converge to zero for $n \geq 3$, and the only possible nonzero differential to or from these groups has domain $E^2_{1,0}(n) \cong \HH_{1}(H_0(\Aut(F,p);R))_n$, which vanishes for $n >4$.  Hence 
$E^2_{-1,1}(n) \cong \HH_{-1}(H_1(\Aut(F,p);R))_n = 0$ for $n >4$.  Similarly the groups $E^r_{0,1}(n)$ converge to zero for $n \geq 5$ and admit no nonzero differentials for $n > 6$. We conclude 
$$ \HH_{-1}(H_1(\Aut(F,p);R))_n = 0 \qquad \text{for $n>4$,} \qquad \text{and} \qquad  \HH_{0}(H_1(\Aut(F,p);R))_n = 0 \qquad \text{for $n>6$}. $$ 
By \autoref{ThmPatztBoundedResolution}, there is a partial resolution $ \cM^1 \to \cM^0 \to H_1(\Aut(F,p);R)$  with $\cM^0$ an induced $\VIC^{\pm}(\Z/p\Z)$--module generated in degree $\leq 4$, and $\cM^0$ an induced $\VIC^{\pm}(\Z/p\Z)$--module generated in degree $\leq 6$.

We proceed by induction. Now assume that $j >1$ and that for all $1\leq i <j$ we have constructed a partial resolution of the $\VIC^{\pm}(\Z/p\Z)$--module $H_i(\Aut(F,p);R)$ $$ \cM^1 \to \cM^0 \to H_i(\Aut(F,p);R)$$ by induced modules with $\cM^{0}$ generated in degree $\leq \left( (\frac{13}{2})3^{2i-3}-\frac12 \right)$ and $\cM^{1}$ generated in degree $\leq \left( (\frac{13}{2})3^{2i-2}-\frac12 \right)$. By \autoref{ThmBoundedResolutionVICU}, we can extend this to a resolution $\cM^* \to H_i(\Aut(F,p);R)$ by induced $\VIC^{\pm}(\Z/p\Z)$--modules with $\cM^k$ generated in degree at most 
$$ \left( 3^{k-1} \left( \left(\frac{13}{2}\right)3^{2i-2}-\frac12 \right) + \frac12(3^{k-1}-1) \right)=  \left(\left(\frac{13}{2}\right)(3^{k-1})(3^{2i-2})-\frac12 \right)  \qquad \text{ for $k\geq 1$.}$$ \autoref{ThmPatztBoundedResolution} then implies that for $k\geq 1$, 
$$E^2_{k,i}(n) \cong  \HH_{k}(H_{i}(\Aut(F,p);R))_n = 0 \qquad \text{for $n >\left(\left(\frac{13}{2}\right)(3^{k})(3^{2i-2})-\frac12\right)$.}$$ 

In particular, for each $2 \leq r \leq j+1$,   $$E^2_{-1+r, j-r+1}(n) \cong \HH_{-1+r}(H_ {j-r+1}(\Aut(F,p);R)))_n =0  \qquad \text{ for $n >\left(\left(\frac{13}{2}\right)\left(3^{2(j-r+1)-2}\right)\left(3^{-1+r}\right)-\frac12\right)$.}$$ Hence, for $r \geq 2$, there are no nonzero differentials to or from $E^r_{-1, j}(n)$ for $r \geq 2$ and  $n >\left((\frac{13}{2})(3^{2j-3})-\frac12\right)$. Since $E^{\infty}_{-1, j}(n) = 0$ in this range, we conclude that 
$$ E^2_{-1, j}(n) =  \HH_{-1}(H_j(\Aut(F,p);R))_n = 0 \quad \text{ for $n >\left(\frac{13}{2}\right)\left(3^{2j-3}\right)-\frac12$.}$$ 

Similarly, for each $2 \leq r \leq j+1$, the group  $E^2_{r, j-r+1}(n) \cong \HH_{r}(H_ {j-r+1}(\Aut(F,p);R)))_n$ vanishes for $n >\left((\frac{13}{2})(3^{2(j-r+1)-2})(3^{r})-\frac12\right)$. This implies that there are no nonzero differentials to or from $E^r_{0, j}(n)$ for $r \geq 2$ and  $n >\left((\frac{13}{2})(3^{2j-2})-\frac12\right)$. Again $E^{\infty}_{0, j}(n) = 0$ in this range, so we conclude that 
$$ E^2_{0, j}(n) =  \HH_{0}(H_j(\Aut(F,p);R))_n = 0 \qquad \text{for $n >\left( \left(\frac{13}{2}\right)(3^{2j-2}) - \frac12 \right)$.}$$ 
By \autoref{ThmPatztBoundedResolution},  $H_j(\Aut(F,p);R)$ is generated in degree $\leq \left((\frac{13}{2})(3^{2j-3}) - \frac12\right)$ and related in degree $\leq \left( (\frac{13}{2})(3^{2j-2})-\frac12 \right)$. This completes the inductive step and concludes the proof. 
\end{proof}

\begin{remark} \label{remark3manifold}

Let $G=\pi_1(P)$ with $P$ an orientable prime $3$-manifold such that $\Mod(P) \twoheadrightarrow \Aut(G)$. Examples of such groups include $\Z$, $\Z/2$, $\Z/4$, $\Z/6$, and $\pi_1(\Sigma_g)$; see the introduction of Hatcher--Wahl \cite{HW}. Many of these groups admit surjections $\phi: G \m \Z/p\Z$ for some prime $p$. Given such a surjection, let $\Aut(G^{*n},\phi)$ denote the kernel of $\Aut(G^{*n})  \m GL_n(\Z/p\Z)$. Here $G^{*n}$ denotes the $n$-fold free product of $G$. An analogous stability result to \autoref{detailedAutp} can be proven for $H_i(\Aut(G^{*n},\phi);R)$ using \cite[Lemma 5.6]{RWW} to establish the analogue of \autoref{connectivityautomorphism} of \autoref{SSaut}.  As it is not known if $H_i(\Aut(G^{*n},\phi);R)$ is finitely generated for $G \neq \Z$, it is unclear if the Noetherian techniques of Putman--Sam \cite[Theorem D]{PS} apply to $H_i(\Aut(G^{*n},\phi);R)$ for $G \neq \Z$.

\end{remark}

\section{Twisted stability results}
\label{TwistedStabilityResults}

\subsection{Tensor products of $\Cat$--modules}

Before we can prove our twisted homological stability theorems, we first must establish some algebraic properties of tensor products of $\VIC^{\U}(\bk)$-- and $\SI(\bk)$--modules. Let $\Cat$ be one of the categories $\SI(\bk)$ or $\VIC^{\U}(\bk)$. Let $\cA$ and $\cB$ be $\Cat$--modules over a commutative ring $R$. Let $\cA \otimes_R \cB$ be the $\Cat$--module defined by the pointwise tensor product, with $$(\cA \otimes_R \cB)_n \cong \cA_n \otimes_R \cB_n$$ and maps $(\cA \otimes_R \cB)_m \m (\cA \otimes_R \cB)_n$ given by the tensor product of the maps $\cA_m \m \cA_n$ with the maps $\cB_m \m \cB_n$.

Our first goal of this section is to determine bounds on the generation and presentation degree of the tensor product $\cA \otimes_R \cB$ in terms of the bounds on the factors $\cA$ and $\cB$.  We begin by recalling some connectivity results from Miller--Patzt--Wilson \cite{MPW} and Mirzaii--van der Kallen \cite{MvdK}.

\begin{definition}
Given a vector space $V$ and subspaces $U$ and $W$, let $\PBC_\bullet(V,U,W)$ be the augmented semi-simplicial set with value on an ordered set $X$ given by \[ \PBC_X(V,U,W)=\{ \; (f,C)  \in \Hom_{\VIC(\bk)}(v(X),V) \; | \; \im(f) \subseteq U, W \subseteq C  \; \} . \] The augmented semi-simplicial structure is induced by composition of ordered sets. 
\end{definition}
\begin{theorem}[Miller--Patzt--Wilson {\cite[Theorem 2.20]{MPW}}] \label{PBCVUW}
For $\bk$ a field, $||\PBC_\bullet(V,U,W)||$ is $\displaystyle \left(\frac{\dim U - \dim W -3}{2}\right)$--connected.
\end{theorem}
In particular, $||\PBC_\bullet(V,U,W)||$ is non-empty if 
$\dim U \geq 1+ \dim W$ and is connected if $\dim U \geq 3+ \dim W$. 

Given a symplectic vector space $V$, following Mirzaii--van der Kallen \cite{MvdK}, Miller--Patzt--Wilson define augmented semi-simplicial sets $\SPB_{\bullet}(V)$  \cite[Definition 2.30]{MPW} and $\MPB_{\bullet}(V)$   \cite[Definition 2.33]{MPW}. We will not define these two objects here, but merely recall the following: given a (not necessarily symplectic) subspace $W \subseteq V$, we obtain an augmented semi-simplicial set $\SPB_{\bullet}(V ) \cap  \Lk^{\MPB(V)}_{\bullet}(W)$ defined on an ordered set $X$ by
$$ \SPB_{X}(V ) \cap  \Lk^{\MPB(V)}_{X}(W) = \{ T \in \Hom_{\SI(\bk)}( s(X), V) \; | \; W\subseteq \im(T)^{\perp} \}.$$

\begin{theorem}[Mirzaii--van der Kallen {\cite[Theorem 7.4]{MvdK}}; see Miller--Patzt--Wilson {\cite[Theorem 2.34]{MPW}}] \label{SPBMPB}
Let $\bk$ be a field. Let $V$ be a symplectic vector space. Let $W$ be a subspace of $V$, and $U$ a maximal symplectic subspace of $W$. Then  $||\SPB_{\bullet}(V ) \cap  \Lk^{\MPB(V)}_{\bullet}(W)||$ is $$\displaystyle \left(\frac{\frac12 \dim V + \frac12 \dim U  -  \dim W-4}{2}\right)\text{--connected.}$$
\end{theorem}

Using \autoref{PBCVUW} and \autoref{SPBMPB}, we will prove the following.

\begin{lemma} \label{tensorOfInduced} Let $\bk$ be a field and $R$ a commutative ring with $\U \subseteq R^{\times}$. If $\Cat=\VIC^{\U}(\bk)$, then $\cM(a) \otimes_R \cM(b)$ has generation degree $\leq a+b+ \min(a,b)$ and presentation degree $\leq a+b+ \min(a,b)+2$. If $\Cat=\SI(\bk)$, then $\cM(a) \otimes_R \cM(b)$ has generation degree $\leq a+b+ \min(a,b)+1$ and presentation degree $\leq a+b+ \min(a,b)+4$. 
\end{lemma}

We will see in \autoref{RemarkBoundsSharp} that the bounds on generation degree in this theorem are sharp when $\C=\VIC(\bk)$, and consequently that the tensor products  $\cM(a) \otimes_R \cM(b)$ of representable $\VIC(\bk)$--modules are not in general induced $\VIC(\bk)$--modules. 

\begin{proof}[Proof of \autoref{tensorOfInduced}]
Let us first consider the case that $\Cat=\VIC(\bk)$. We may assume $a,b>0$ since otherwise $$\cM(a) \otimes_R \cM(0) = \cM(a)$$ and the result is trivial. Let $X$ be an ordered set of size $i+1$ and let $V$ be a vector space of dimension $n$.
By \autoref{ThmPatztBoundedResolution}, it suffices to show 
\begin{align*}
\HH_{-1}\big(\cM(a) \otimes_R \cM(b)\big)_V=0 \qquad & \text{ for $n> a+b+ \min(a,b)$, and that } \\ 
 \HH_{0}\big(\cM(a) \otimes \cM(b)\big)_V=0 \qquad & \text{ for $n> a+b+ \min(a,b)+2$. }
 \end{align*}

We have
\begin{align*}
\CC_X\big(M(a)  \otimes_R & M(b)\big)_V    = \bigoplus_{(T,C) \in \Hom_{\VIC(\bk)}(v(X),V)} \big(\cM(a) \otimes_R \cM(b) \big)_{C}  \\[1em]
&\cong  \bigoplus_{(T,C) \in   \Hom_{\VIC(\bk)}(v(X),V)} R \left [ \Hom_{\VIC(\bk)}(\bk^a,C)  \right ] \otimes_R R \left [ \Hom_{\VIC(\bk)}(\bk^b,C)  \right ] \\[1em] 
&\cong R \left [ \bigsqcup_{(T,C) \in   \Hom_{\VIC(\bk)}(v(X),V)}  \Hom_{\VIC(\bk)}(\bk^a,C)   \times \Hom_{\VIC(\bk)}(\bk^b,C)  \right ] .
\end{align*}
An element in the set $$ \bigsqcup_{(T,C) \in   \Hom_{\VIC(\bk)}(v(X),V)}  \Hom_{\VIC(\bk)}(\bk^a,C)   \times \Hom_{\VIC(\bk)}(\bk^b,C)  $$ is a triple $$ \Big( (T,C), (T_a,C_a), (T_b,C_b) \Big)$$ 
with $$T: v(X) \to V, \quad V \cong C \oplus \im(T),  \qquad T_a: \bk^a \to C, \quad C \cong C_a \oplus \im(T_a), \qquad T_b: \bk^b \to C, \quad C \cong C_b \oplus \im(T_b).$$ When $X = \{0, 1, \ldots, i\}$ and $v(X)=$span$_{\bk}(e_0, e_1\ldots, e_i)$, then the face map $d_j$ maps the above summand to the summand indexed as follows. Let $T_{\setminus j}$ denote the restriction of $T$ to span$_{\bk}(e_0, e_1\ldots, \hat{e_j}, \ldots e_i)$. Then the image under $d_j$ is the summand associated to the triple
$$ \Big(  \big(T_{\setminus j},  (C\oplus \mathrm{span}(T(e_j))\big), \quad \big(T_a, (C_a\oplus \mathrm{span}(T(e_j))\big),  \quad \big(T_b, ( C_b\oplus \mathrm{span}(T(e_j))\big) \Big).$$
We can re-index our set to identify $ \Big( (T,C), (T_a,C_a), (T_b,C_b) \Big)$ with the following triple
$$ \Big( (T,C), (T_a,(C_a\oplus \im(T))), (T_b,(C_b \oplus \im(T))) \Big)$$  
in $$ \Hom_{\VIC(\bk)}(v(X),V) \times   \Hom_{\VIC(\bk)}(\bk^a,V) \times \Hom_{\VIC(\bk)}(\bk^b,V) $$ 
satisfying $$  \im(T) \subseteq \Big( (C_a+\im(T)) \cap (C_b+\im(T)) \Big) \qquad \text{and}\quad \Big( \im(T_a) +\im(T_b) \Big) \subseteq C. $$
The face map $d_j$ now acts only on $(T,C)$ while fixing the pairs $(T_a, (C_a\oplus \im(T)))$ and $(T_b,(C_b \oplus \im(T)))$. 
Conversely, we can recover $\Big( (T,C), (T_a,C_a), (T_b,C_b) \Big)$ from this triple using the equalities 
$$ C_a = C \cap (\im(T) \oplus C_a) \qquad C_b = C \cap (\im(T) \oplus C_b); $$ 
see \cite[Proposition 2.9 (vi)]{MPW}.  Thus, we obtain the following isomorphism of augmented semi-simplicial $R$--modules. 
\begin{align*}
&\CC_{X}\big(\cM(a)  \otimes_R  \cM(b)\big)_V      \\
&\cong R \left [ \bigsqcup_{\big( (T_a,C_a), (T_b,C_b) \big) \in   \Hom_{\VIC(\bk)}(\bk^a,V) \times \Hom_{\VIC(\bk)}(\bk^b,V)}  \PBC_{X} \Big (V,C_a \cap C_b, \im(T_a) + \im(T_b)  \Big ) \right ] .
   \end{align*}  
   
Suppose without loss of generality that $a \geq b$, and fix a pair $\big( (T_a,C_a), (T_b,C_b) \big)$. Because $$V = C_a \oplus \im(T_a), \qquad \qquad C_a \cap C_b \subseteq C_a \qquad \text{and} \qquad \im(T_a) \subseteq \big(\im(T_a) + \im(T_b)\big),$$ \cite[Lemma 2.18]{MPW} implies that 
$$ \PBC_{\bullet} \Big (V,C_a \cap C_b, \im(T_a) + \im(T_b) \Big) \cong \PBC_{\bullet} \Big (C_a ,C_a \cap C_b,  \big(\im(T_a) + \im(T_b)\big) \cap C_a \Big).$$

   Taking homology yields 
   \begin{align*}
  & \HH_i\Big(\cM(a) \otimes_R \cM(b)\Big)_V  \\ 
   \cong &   \bigoplus_{\big( (T_a,C_a), (T_b,C_b) \big ) \in   \Hom_{\VIC(\bk)}(\bk^a,V) \times \Hom_{\VIC(\bk)}(\bk^b,V)}   \widetilde H_i\Big(\big| \big| \PBC_{\bullet} \Big (C_a ,C_a \cap C_b,  \big(\im(T_a) + \im(T_b)\big) \cap C_a \Big) \big|\big| ;R\Big) .  
   \end{align*}
   
   Observe that \[ \dim (C_a \cap C_b) \; \; \geq \;\;n-a-b \qquad  \text{ and } \qquad  \dim \left( \big(\im(T_a) + \im(T_b)\big) \cap C_a \right) \; \; \leq \;\; b \;\; =\;\; \min(a,b).\] By \autoref{PBCVUW}, \[\widetilde H_{-1}\Big(\big| \big| \PBC_\bullet \Big (V,C_a \cap C_b, \im(T_a) + \im(T_b)  \Big ) \big|\big| ;R \Big) \cong 0 \text{ for } n>a+b+\min(a,b)\] and \[\widetilde H_{0}\Big(\big| \big| \PBC_\bullet \Big (V,C_a \cap C_b, \im(T_a) + \im(T_b)  \Big ) \big|\big|;R \Big) \cong 0 \text{ for }n>a+b+\min(a,b)+2.\] 
   The claim now follows for $\Cat=\VIC(\bk)$. 
   
   Now suppose that $\Cat=\VIC^{\U}(\bk)$, and again we may assume that $a,b >0$. Recall that  $$\Hom_{\VIC}(R^d, R^n) = \Hom_{\VIC^{\U}}(R^d, R^n)  \qquad \text{whenever $d \neq n$}.$$ 
   Thus the complexes  $\CC_{\bullet}\big(\cM(a)  \otimes_R  \cM(b)\big)_n $ associated $\VIC(\bk)$ and to $\VIC^{\U}(\bk)$ have the same $p$--chains for $p\leq 0$ when $n> a+b+\min(a,b)$, and for $p \leq 2$ when $n> a+b+\min(a,b)+2$. Hence, the results proved for $\VIC(\bk)$ in homological degree $-1$ and $0$ also hold for $\VIC^{\U}(\bk)$. 
   
   Finally, consider $\Cat=\SI(\bk)$, and again let $a,b>0$. Let $X$ be an ordered set of size $i+1$ and let $V$ be a  symplectic vector space of dimension $2n$. Then 
\begin{align*}
\CC_X\big(\cM(a)  \otimes_R & \cM(b)\big)_V    = \bigoplus_{T \in \Hom_{\VIC(\bk)}(s(X),V)} \big(\cM(a) \otimes_R \cM(b) \big)_{\im(T)^{\perp}} \\[1em]
&\cong R \left [ \bigsqcup_{T \in \Hom_{\VIC(\bk)}(s(X),V)}  \Hom_{\SI(\bk)}(\bk^{2a},\im(T)^{\perp})   \times \Hom_{\VIC(\bk)}(\bk^{2b},\im(T)^{\perp})  \right ] .
\end{align*}
Again we have the $R$--vector space on triples of symplectic maps $(T, T_a, T_b)$ with $$ T: s(X) \to V, \qquad T_a: \bk^{2a} \to \im(T)^{\perp}, \qquad T_b :\bk^{2b} \to \im(T)^{\perp}.$$ 
Equivalently, this is the space of triples  $(T, T_a, T_b)$ with $$ T_a: \bk^{2a} \to V, \qquad T_b :\bk^{2b} \to V,  \qquad T: s(X) \to \big(\im(T_a) + \im(T_b)\big)^{\perp}.$$ 
%
We note that $\big(\im(T_a) + \im(T_b)\big)^{\perp}$ need not be a symplectic subspace. Thus, in the notation of Miller--Patzt--Wilson \cite{MPW}, we have an isomorphism of semi-simplicial $R$--modules
\begin{align*}
\CC_{\bullet}\big(\cM(a)  \otimes_R & \cM(b)\big)_V   \\
&\cong R \left [ \bigsqcup_{(T_a, T_b) \in \Hom_{\SI(\bk)}(\bk^{2a},V)   \times \Hom_{\VIC(\bk)}(\bk^{2b},V) }  \SPB_{\bullet}(V)\cap \Lk_{\bullet}^{\MPB(V)}\Big(\im(T_a) +\im(T_b)\Big)  \right ] .\\
\end{align*}

Suppose that $a \geq b$. Then $(\im(T_a) +\im(T_b))$ has dimension at most $2a+2b$, and contains the symplectic subspace $\im(T_a)$ of dimension $2a$.  By \autoref{SPBMPB}, then, the homology groups 
$$H_i\Big(\CC_{\bullet}\big(\cM(a)  \otimes_R  \cM(b)\big)_V\Big) = 0$$ for $$ i \leq \left( \frac{n +a -(2a+2b) -4}{2} \right) = \left( \frac{n-a-b - \min(a,b) -4}{2} \right).$$
In particular,  
\begin{align*}
\HH_{-1}\Big(\cM(a)  \otimes_R \cM(b)\Big)_n = 0 \qquad &  \text{for } n > a+b+\min(a,b)+1, \quad \text{and}\\
\HH_{0}\Big(\cM(a)  \otimes_R   \cM(b) \Big)_n = 0 \qquad &  \text{for }  n > a+b+\min(a,b)+3.
\end{align*}
Thus by \autoref{ThmPatztBoundedResolution}, we can conclude that the $\SI(\bk)$--module $\cM(a) \otimes_R \cM(b)$ is generated in degree $\leq a+b+\min(a,b)+1$ and presented in degree $\leq a+b+\min(a,b)+4$. 
\end{proof}

The full statement of Miller--Patzt--Wilson \cite[Theorem 2.20]{MPW} also applies to the case when $\bk$ is a PID. A similar argument to our proof of \autoref{tensorOfInduced} would 
give an analogue of \autoref{tensorOfInduced} in this case, with a worse stable range. 

\begin{remark} \label{RemarkBoundsSharp} Let $\Cat$ be the category $\VIC^{\U}(\bk)$ or $\SI(\bk)$ for $\bk$ a field, and let $R$ be a commutative ring. We remark that, in contrast to the case of $\FI$--modules, the tensor product $\cM(a) \otimes_R \cM(b)$ of representable $\Cat$--modules over $R$ is not generated in degree $\leq (a+b)$.   Suppose that $a\geq b$. First let $\Cat=\VIC(\bk)$. We can show that the bounds on the generation degree given in \autoref{tensorOfInduced} are sharp. Let $e_1, \ldots, e_n$ denote the standard $\bk$--basis for the object $\bk^n$ of $\VIC(\bk)$.    Consider an $R$--basis element
$$(f, C_f) \otimes (g, C_g) \in \cM(a)_n \otimes_R \cM(b)_n$$
for $n=a+2b$ with 
\begin{align*}
\im(f)=& \; \mathrm{span}(e_1, e_2, \ldots, e_a) \qquad \\
 C_f =&\; \mathrm{span}(e_1+e_{a+1}, e_2+e_{a+2}, \ldots, e_b+e_{a+b}, \; \;  e_{a+b+1}, \ldots, e_{a+2b}); \\
\im(g)=&\; \mathrm{span}(e_{a+1}, e_{a+2}, \ldots, e_{a+b}) \qquad\\  
C_g =&\; \mathrm{span}(e_1+e_{a+1}, e_2+e_{a+2}, \ldots, e_b+e_{a+b},  \; \; e_{b+1}, \ldots, e_a, \; \; \\ &  e_{a+1}+e_{a+b+1}, e_{a+2}+e_{a+b+2}, \ldots, e_{a+b}+e_{a+2b}).
\end{align*}
Since $$C_g \cap C_g = \mathrm{span}(e_1+e_{a+1}, e_2+e_{a+2}, \ldots, e_b+e_{a+b}) \subseteq \big(\mathrm{im}(f) +\mathrm{im}(g)\big),$$  it follows that $(f, C_f) \otimes (g, C_g)$ is not in the image of $\cM(a)_n \otimes_R \cM(b)_n$ for any $n<a+2b$. We make an additional observation pointed out to us by Rohit Nagpal: when $\bk$ is finite, the dimension of $\cM(a)_n \otimes_R \cM(b)_n$ grows too slowly in $n$ for $\cM(a) \otimes_R \cM(b)$ to contain induced representations of the form $\cM(W)$ with $W$ supported in degree $>(a+b)$. This implies that, when $a,b >0$, the $\VIC(\bk)$--module $\cM(a) \otimes_R \cM(b)$ is not an induced module.\\
 Similarly, let $a \geq b>0$, and consider the $\SI(\bk)$--module $\cM(a) \otimes_R \cM(b)$. We will show that it too has generators in degree $n=a+2b$. Let $v_1, w_1, v_2, w_2, \ldots, v_n, w_n$ denote the standard symplectic basis for $\bk^{2n}$. Consider a basis element 
$$f\otimes g \in \cM(a)_n \otimes_R \cM(b)_n$$
for $n=a+2b$ with 
\begin{align*}
\im(f) &= \mathrm{span}(v_1, w_1, \ldots, v_a, w_a) \qquad \text{and} \\ 
\im(g) &= \mathrm{span}(v_1+v_{a+1}, w_1+v_{a+2}, v_2+v_{a+3},\ldots, v_b+v_{a+2b-1}, w_b+v_{a+2b}).
\end{align*}
Then $$\Big(\im(f) + \im(g)\Big)= \mathrm{span}(v_1, w_1, v_2, w_2, \ldots, v_a, w_a, \; v_{a+1}, v_{a+2}, \ldots, v_{a+2b})$$ is not contained in any proper symplectic subspace, and so $f\otimes g$ is not in the image of $\cM(a)_n \otimes_R \cM(b)_n$  for any $n<a+2b$. 
\end{remark}

 We can now use the results of \autoref{tensorOfInduced} to establish bounds on the generation and presentation degree of arbitrary tensor products.

\begin{proposition} \label{ResolvingTensors} Let $\Cat$ be $\SI(\bk)$ or $\VIC^{\U}(\bk)$. Let $\cA$ and $\cB$ be $\Cat$--modules over a commutative ring $R$ with generation degrees $\leq d_\cA$ and $\leq d_\cB$ respectively, and relation degrees $\leq r_\cA$ and $ \leq r_\cB$ respectively.  If $\Cat=\VIC^{\U}(\bk)$, then $\cA \otimes_R \cB$ has generation degree $$\leq \Big(d_\cA+d_\cB+\min(d_{\cA}, d_{\cB}) \Big)$$ and relation degree $$\leq \max\Big(d_\cA+r_\cB+\min(d_\cA, r_\cB), \; r_\cA+d_\cB+\min(r_\cA,d_\cB),\; d_\cA+d_\cB+\min(d_{\cA}, d_{\cB})+2\Big).$$ 
If $\Cat=\VIC^{\U}(\bk)$, then $\cA \otimes_R \cB$ has generation degree $$\leq \Big(d_\cA+d_\cB+\min(d_{\cA}, d_{\cB}) +1 \Big)$$ and relation degree $$\leq \max\Big(d_\cA+r_\cB+\min(d_\cA, r_\cB) +1, \; r_\cA+d_\cB+\min(r_\cA,d_\cB) +1,\; d_\cA+d_\cB+\min(d_{\cA}, d_{\cB})+4\Big).$$

\end{proposition}

\begin{proof}
Let $\cP_{\bullet}^\cA$ and $\cP_{\bullet}^\cB$ be a resolutions of $\cA$ and $\cB$ respectively by induced $\Cat$--modules with $\cP^{\cA}_0$, $\cP^{\cB}_0$, $\cP^{\cA}_1$, and $\cP^{\cB}_0$ generated in degree $\leq d_\cA, d_\cB, r_\cA, r_\cB$ respectively. Take the total complex associated to the double complex $\cP_{\bullet}^\cA \otimes_R \cP_{\bullet}^\cB$. The total complex is exact because the rows and columns of the double complex are. Thus, we have a resolution
$$\ldots \to \Big( (\cP_{0}^A \otimes_R \cP_{1}^\cB) \oplus (\cP_{1}^\cA \otimes_R \cP_{0}^\cB) \Big) \to \Big( \cP_0^\cA \otimes_R \cP_0^\cB \Big) \to \cA \otimes_R \cB.$$ Define the degree of a $\Cat$--module $\cC$ to be the largest number $n$ such that $\cC_n \neq 0$ and denote this by $\deg \cC$. By considering the hyperhomology spectral sequence associated to this resolution and the functor $H_0^\Cat$, we see that 
\[\deg H_0^\Cat(\cA \otimes_R \cB)  \leq \deg H_0^\Cat \left (\cP_0^\cA \otimes_R \cP_0^\cB \right ) \] and \[\deg H_1^\Cat(\cA \otimes_R \cB)  \leq \max \Big( \deg H_0^\Cat \left ( (\cP_{0}^A \otimes_R \cP_{1}^\cB) \oplus (\cP_{1}^\cA \otimes_R \cP_{0}^\cB) \right)  , \; \deg H_1^\Cat \left (        \cP_0^\cA \otimes_R \cP_0^\cB    \right )       \Big ). \]  
\autoref{tensorOfInduced} then implies that for  $\Cat=\VIC^{\U}(\bk)$, 
\begin{align*}
\deg H_0^\Cat \left (\cP_0^\cA \otimes_R \cP_0^\cB \right )  &\leq d_\cA+d_\cB+\min(d_{\cA}, d_{\cB}),\\ 
\deg H_0^\Cat \Big( (\cP_{0}^A \otimes_R \cP_{1}^\cB) \oplus (\cP_{1}^\cA \otimes_R \cP_{0}^\cB) \Big) & \leq \max\Big(d_\cA+r_\cB+\min(d_\cA, r_\cB), r_\cA+d_\cB+\min(r_\cA,d_\cB)\Big), \text{ and } \\ 
\deg H_1^\Cat \left (\cP_0^\cA \otimes_R \cP_0^\cB \right )  &\leq d_\cA+d_\cB+\min(d_{\cA}, d_{\cB})+2.
\end{align*} 
 For  $\Cat=\SI(\bk)$, 
\begin{align*}
\deg H_0^\Cat \left (\cP_0^\cA \otimes_R \cP_0^\cB \right )  &\leq d_\cA+d_\cB+\min(d_{\cA}, d_{\cB}) +1,\\ 
\deg H_0^\Cat \Big( (\cP_{0}^A \otimes_R \cP_{1}^\cB) \oplus (\cP_{1}^\cA \otimes_R \cP_{0}^\cB) \Big) +1 & \leq \max\Big(d_\cA+r_\cB+\min(d_\cA, r_\cB), r_\cA+d_\cB+\min(r_\cA,d_\cB)+1\Big), \text{ and } \\ 
\deg H_1^\Cat \left (\cP_0^\cA \otimes_R \cP_0^\cB \right )  &\leq d_\cA+d_\cB+\min(d_{\cA}, d_{\cB})+4.
\end{align*} 
The claim now follows from \autoref{presentLemma} wich relates vanishing of $H_0^{\Cat}$ and $H_1^{\Cat}$ to generation and relation degree.

\end{proof}

\subsection{Homological stability with twisted coefficients}

In this subsection, we prove \autoref{MCGstab} and \autoref{AutFnStab}.

An inclusion of a surface $\Sigma_{g,1}$ into $\Sigma_{g+1,1}$ induces a map $\Mod(\Sigma_{g,1}) \m \Mod(\Sigma_{g+1,1})$. If $\A$ is an $\SI(\Z/p\Z)$--module, then this inclusion map gives a map: \[H_i(\Mod(\Sigma_{g,1});\A_g) \m H_i(\Mod(\Sigma_{g+1,1});\A_{g+1}).\] See Putman--Sam \cite[Section 4]{PS} for more details on this and the analogous construction in the case of $\Aut(F_n)$ and $\VIC^\pm(\Z/p\Z)$--modules.

\begin{proof}[Proof of \autoref{MCGstab}]
Let $\G_g$ denote $\Sp_{2g}(\Z/p\Z)$, let $R$ be a field of characteristic zero, and let $\A$ be an $\SI(\Z/p\Z)$--module over $R$ with generation degree $d$ and relation degree $r$. Given a group $Q$, let $C_*(Q;R)$ denote a chain complex computing group homology of $Q$ with coefficients in $R$. All $R[\G_g]$--modules are flat, so the operation of tensoring over $R[\G_g]$ commutes with taking homology. We have \begin{align*} H_i(\Mod(\Sigma_{g,1});\cA_g) \cong& H_i( C_{*}(\Mod(\Sigma_{g,1},p);R) \otimes_{R[\G_g]} \cA_g)\\
\cong& H_i( \Mod(\Sigma_{g,1},p);R) \otimes_{R[\G_g]} \cA_g\\
\cong& \left(H_i( \Mod(\Sigma_{g,1},p);R) \otimes_R \cA_g \right)_{\G_g}  .
\end{align*} 

Let $D_i$ denote the generation degree of $H_i( \Mod(\Sigma,p);R)$ and $R_i$ denote the relation degree. By \autoref{ResolvingTensors}, $H_i( \Mod(\Sigma,p);R) \otimes_R \A_g$ has generation degree $\leq D_i+d + \min(D_i, d)$ and relation degree $$\leq \max(D_i+r + \min(D_i, r)+1,R_i+d+\min(R_i, d)+1, D_i+d + \min(D_i, d) +4).$$ By \autoref{BoundingCoinvariants}, 
\[\left(H_i( \Mod(\Sigma_{g,1},p);R) \otimes_R \A_g \right)_{\G_g} \cong \left(H_i( \Mod(\Sigma_{g+1,1},p);R) \otimes_R \A_g \right)_{\G_{g+1}} \] for $g \geq \max(D_i+r + \min(D_i, r)+1,R_i+d+\min(R_i, d)+1, D_i+d + \min(D_i, d) +4)$. The claim now follows from the bounds on $D_i$ and $R_i$ from \autoref{detailedModp}.
\end{proof}

\begin{proof}[Proof of \autoref{AutFnStab}]
The proof is the same as the proof of \autoref{MCGstab} except we use the bounds from \autoref{detailedAutp}.
\end{proof}


\bibliographystyle{amsalpha}
\bibliography{Levelp}

\end{document}